\documentclass[11pt]{article}
\usepackage[latin1]{inputenc}
\usepackage{amsmath,amsthm,amssymb}
\usepackage{amsfonts}
\usepackage{amsmath,amsthm,amssymb,amscd}
\usepackage{latexsym}
\usepackage{color}
\usepackage{graphicx}
\usepackage{mathrsfs}
\usepackage{cite}

\usepackage{color,enumitem,graphicx}
\usepackage[colorlinks=true,urlcolor=black,
citecolor=black,linkcolor=black,linktocpage,pdfpagelabels,
bookmarksnumbered,bookmarksopen]{hyperref}

\makeatletter \@addtoreset{equation}{section} \makeatother

\newtheorem{theorem}{Theorem}[section]
\newtheorem{definition}{Definition}[section]
\newtheorem{proposition}{Proposition}[section]
\newtheorem{lemma}{Lemma}[section]
\newtheorem{remark}{Remark}[section]

\allowdisplaybreaks

\begin{document}
\title{Multiple blowing-up solutions for asymptotically critical Lane-Emden systems on Riemannian manifolds}

\author{Wenjing Chen\footnote{Corresponding author.}\ \footnote{E-mail address:\, {\tt wjchen@swu.edu.cn} (W. Chen), {\tt zxwangmath@163.com} (Z. Wang).}\  \ and Zexi Wang\\
\footnotesize  School of Mathematics and Statistics, Southwest University,
Chongqing, 400715, P.R. China}

\date{ }
\maketitle

\begin{abstract}
{Let $(\mathcal{M},g)$ be a smooth compact Riemannian manifold of dimension $N\geq 8$. We are concerned with the following elliptic system
\begin{align*}
 \left\{
  \begin{array}{ll}
  -\Delta_g u+h(x)u=v^{p-\alpha \varepsilon},
  \ \  &\mbox{in}\ \mathcal{M},\\
   -\Delta_g v+h(x)v=u^{q-\beta \varepsilon},
  \ \  &\mbox{in}\ \mathcal{M},\\
  u,v>0,
  \ \  &\mbox{in}\ \mathcal{M},
    \end{array}
    \right.
  \end{align*}
where $\Delta _g=div_g \nabla$ is the Laplace-Beltrami operator on $\mathcal{M}$, $h(x)$ is a $C^1$-function on $\mathcal{M}$, $\varepsilon>0$ is a small parameter, $\alpha,\beta>0$ are real numbers, $(p,q)\in (1,+\infty)\times (1,+\infty)$ satisfies $\frac{1}{p+1}+\frac{1}{q+1}=\frac{N-2}{N}$. Using the Lyapunov-Schmidt reduction method, we obtain the existence of multiple blowing-up solutions for the above problem.}



\end{abstract}

\section{Introduction}
Let $(\mathcal{M},g)$ be a smooth compact Riemannian manifold of dimension $N\geq 8$, where $g$ denotes the metric tensor. We consider the following elliptic system
\begin{align}\label{pro}
 \left\{
  \begin{array}{ll}
  -\Delta_g u+h(x)u=v^{p-\alpha \varepsilon},
  \ \  &\mbox{in}\ \mathcal{M},\\
   -\Delta_g v+h(x)v=u^{q-\beta \varepsilon},
  \ \  &\mbox{in}\ \mathcal{M},\\
  u,v>0,
  \ \  &\mbox{in}\ \mathcal{M},
    \end{array}
    \right.
  \end{align}
where $\Delta _g=div_g \nabla$ is the Laplace-Beltrami operator on $\mathcal{M}$, $h(x)$ is a $C^1$-function on $\mathcal{M}$, $\varepsilon>0$ is a small parameter, $\alpha,\beta>0$ are real numbers, $(p,q)\in (1,+\infty)\times (1,+\infty)$ is a pair of numbers lying on the {\em critical hyperbola}
 \begin{equation}\label{ch}
    \frac{1}{p+1}+\frac{1}{q+1}=\frac{N-2}{N}.
  \end{equation}
Without loss of generality, we assume that $1<p\leq \frac{N+2}{N-2}\leq q$. Moreover, by \eqref{ch}, we have $p>\frac{2}{N-2}$.

In the case $u=v$, $p=q$ and $\alpha=\beta=1$, system \eqref{pro} is reduced to the following equation
 \begin{equation}\label{yamabe}
   -\Delta_g u+h(x)u=u^{2^*-1-\varepsilon},\quad u>0,\quad \text{in $\mathcal{M}$},
 \end{equation}
where $N\geq3$, $2^*=\frac{2N}{N-2}$, $\varepsilon\in \mathbb{R}$ is a small parameter. If $h(x)=\frac{N-2}{4(N-1)}Scal_g$, where $Scal_g$ is the scalar curvature of the manifold, equation \eqref{yamabe} is intensively studied as the well-known {\em Yamabe problem} whose positive solutions $u$ are such the scalar curvature of the conformal metric $u^{2^*-2}g$ is constant, see e.g. \cite{Y,Au,Sch,Tru}. If $h(x)\neq\frac{N-2}{4(N-1)}Scal_g$, Micheletti et al. \cite{MPV} first proved that \eqref{yamabe} has a single blowing-up solution, provided the graph of $h(x)$ is distinct at some point from the graph of $\frac{N-2}{4(N-1)}Scal_g$. Here, we say that a family of solutions $u_\varepsilon$ of \eqref{yamabe} blows up at a point $\xi_0\in \mathcal{M}$ if there exists a family of points $\xi_\varepsilon \in \mathcal{M}$ such that $\xi_\varepsilon\rightarrow \xi_0$ and $u_\varepsilon(\xi_\varepsilon)\rightarrow+\infty$ as $\varepsilon\rightarrow0$. Soon after, Deng \cite{Deng} considered the existence of multiple blowing-up solutions which are separate from each other for \eqref{yamabe}. Chen \cite{Chen} discovered the existence of clustered solutions which concentrate at one point in $\mathcal{M}$ for \eqref{yamabe}. Moreover, Sign-changing bubble tower solutions of \eqref{yamabe} have been established in \cite{CK,PV}.
 For more related results about \eqref{yamabe}, we refer the readers to \cite{DMW,DPV,RV,GMP} and references therein.

 Now, we return to the following elliptic system
 \begin{align}\label{back}
 \left\{
  \begin{array}{ll}
  -\Delta u=|v|^{p-1}v,
  \quad  \mbox{in}\ \Omega,\\
   -\Delta  v=|u|^{q-1}u,
  \quad  \mbox{in}\ \Omega,\\
   (u,v)\in \mathcal{X}_{p,q}(\Omega),
    \end{array}
    \right.
  \end{align}
called the Lane-Emden system, where $N\geq3$, $(p,q)$ satisfies \eqref{ch}, $\Omega$ is either a smooth bounded domain or $\mathbb{R}^N$, and $\mathcal{X}_{p,q}(\Omega)=\dot{W}^{2,\frac{p+1}{p}}(\Omega)\times \dot{W}^{2,\frac{q+1}{q}}(\Omega)$. System \eqref{back} has received remarkable attention for decades.  When $\Omega=\mathbb{R}^N$, by the Sobolev embedding theorem, there holds
\begin{equation*}
  \dot{W}^{2,\frac{p+1}{p}}(\mathbb{R}^N)\hookrightarrow \dot{W}^{1,p^*}(\mathbb{R}^N)\hookrightarrow L^{q+1}(\mathbb{R}^N),\quad
  \dot{W}^{2,\frac{q+1}{q}}(\mathbb{R}^N)\hookrightarrow \dot{W}^{1,q^*}(\mathbb{R}^N)\hookrightarrow L^{p+1}(\mathbb{R}^N),
\end{equation*}
with
\begin{equation*}
  \frac{1}{p^*}=\frac{p}{p+1}-\frac{1}{N}=\frac{1}{q+1}+\frac{1}{N},\quad \frac{1}{q^*}=\frac{q}{q+1}-\frac{1}{N}=\frac{1}{p+1}+\frac{1}{N}.
\end{equation*}
Thus the following energy functional is well defined in $\mathcal{X}_{p,q}(\mathbb{R}^N)$:
\begin{equation*}
  \mathcal{J}(u,v)=\int\limits_{\mathbb{R}^N}\nabla u \cdot \nabla vdz-\frac{1}{p+1}\int\limits_{\mathbb{R}^N}|v|^{p+1}dz-\frac{1}{q+1}\int\limits_{\mathbb{R}^N}|u|^{q+1}dz.
\end{equation*}
By applying the concentration compactness principle, Lions \cite{Lions} found a positive least energy solution to \eqref{back} in $\mathcal{X}_{p,q}(\mathbb{R}^N)$, which is radially symmetric and radially decreasing. Moreover, Wang \cite{Wang} and Hulshof and Van der Vorst \cite{HV} independently proved that the uniqueness of the
positive least energy solution $(U_{1,0}(x),V_{1,0}(z))\in \mathcal{X}_{p,q}(\mathbb{R}^N)$, and all the positive least energy solutions $\big(U_{\delta,\xi}(z),V_{\delta,\xi}(x)\big)$ given by
\begin{equation*}
  \big(U_{\delta,\xi}(z),V_{\delta,\xi}(z)\big)=\big(\delta^{-\frac{N}{q+1}}U_{1,0}(\delta^{-1}(z-\xi)),\delta^{-\frac{N}{p+1}}V_{1,0}(\delta^{-1}(z-\xi))\big),\quad \text{for any $\delta>0$, $\xi\in \mathbb{R}^N$}.
\end{equation*}
Frank et al. \cite{FKP} established the non-degeneracy of \eqref{back} at each least energy solution, that is, the linearized system around a least energy solution has precisely the $(N+1)$-dimensional spaces of solutions in $\mathcal{X}_{p,q}(\mathbb{R}^N)$. Furthermore, by using the Lyapunov-Schmidt reduction method and the non-degeneracy result obtained in \cite{FKP}, Guo et al. \cite{GLP1} established the existence and non-degeneracy of multiple blowing-up solutions to \eqref{back} with two potentials.
For more investigations of system \eqref{back} with $\Omega=\mathbb{R}^N$, we can see \cite{GM,CS}.

If $\Omega$ is a smooth bounded domain, much attention has been paid to study \eqref{back}. Kim and Pistoia \cite{KP} first built multiple blowing-up solutions for the Lane-Emden system
\begin{align}\label{kp}
 \left\{
  \begin{array}{ll}
  -\Delta u=|v|^{p-1}v+\varepsilon(\alpha u+\beta_1 v),
  \quad  &\mbox{in}\ \Omega,\\
   -\Delta  v=|u|^{q-1}u+\varepsilon(\alpha v+\beta_2 u),
  \quad  &\mbox{in}\ \Omega,\\
   u,v=0,\quad &\mbox{on}\ \partial \Omega,
    \end{array}
    \right.
  \end{align}
where $N\geq8$, $\varepsilon>0$, $\alpha,\beta_1,\beta_2\in \mathbb{R}$, $1<p<\frac{N-1}{N-2}$, and $(p,q)$ satisfies \eqref{ch}. Furthermore, using the local Pohozaev identities for the system, Guo et al. \cite{GHP} proved the non-degeneracy of the blowing-up solutions to \eqref{kp} constructed in \cite{KP}.  Jin and Kim \cite{JK} studied the Coron's problem for the critical Lane-Emden system, and established the existence, qualitative properties of positive solutions.
More recently, inspired by \cite{PST1}, Guo and Peng \cite{GP} considered sign-changing solutions to the sightly supercritical Lane-Emden system with Neumann boundary conditions.
For more classical results regarding Hamiltonian systems in bounded domains,
the readers may refer to \cite{CK1,PST,GLP2,BMT,KM} for a good survey.

As far as we know, no existence result for the system \eqref{pro}-\eqref{ch} in the literature. Therefore, it is natural to ask that if the system possesses solutions on a smooth compact Riemannian manifold.  Motivated by \cite{KP} and \cite{MPV}, in this paper, we give  an affirmative answer for this question.

To state our main result, we first recall some definitions and results.

\begin{definition}
{\rm For $k\geq2$ to be a positive integer, let $(u_\varepsilon,v_\varepsilon)$ be a family of solutions of \eqref{pro}-\eqref{ch}, we say that
$(u_\varepsilon,v_\varepsilon)$ blows up and concentrates at point $\bar{\xi^0}=(\xi_1^0,\xi_2^0,\cdots, \xi_k^0)\in \mathcal{M}^k$ if there exists $\bar{\xi}^\varepsilon=(\xi_1^\varepsilon,\xi_2^\varepsilon,\cdots, \xi_k^\varepsilon)\in \mathcal{M}^k$ and $(\delta_1^\varepsilon,\delta_2^\varepsilon,\cdots, \delta_k^\varepsilon)\in (\mathbb{R}^+)^k$ such that
$\xi_j^\varepsilon\rightarrow \xi_j^0$ and $\delta_j^\varepsilon\rightarrow0$ as $\varepsilon\rightarrow0$ for $j=1,2,\cdots,k$, and
\begin{equation*}
  \Big\|(u_\varepsilon,v_\varepsilon)-\Big(\sum\limits_{j=1}^kW_{\delta_j^\varepsilon,\xi_j^\varepsilon},
  \sum\limits_{j=1}^kH_{\delta_j^\varepsilon,\xi_j^\varepsilon}\Big)\Big\|\rightarrow0,\quad \text{as $\varepsilon\rightarrow0$},
\end{equation*}
where $\|\cdot\|$ and $(W_{\delta,\xi},H_{\delta,\xi})$ are defined in \eqref{norm} and \eqref{fun1}.}
\end{definition}
\begin{definition}\cite[Definition 0.1]{Li}
{\rm Let $f\in C^1(\mathcal{M},\mathbb{R})$, for any given integer $k\geq2$, set $\bar{\xi}=(\xi_1,\xi_2,\cdots, \xi_k)$, let $\mathcal{C}_1,\mathcal{C}_2,\cdots, \mathcal{C}_k\subset \mathcal{M}$ be $k$ mutually disjoint closed subsets of critical points of $f$, we say that $(\mathcal{C}_1,\mathcal{C}_2,\cdots, \mathcal{C}_k)\\
\subset \mathcal{M}^k$ is a
 $C^1$-stable critical set of function $F(\bar{\xi}):=\sum\limits_{j=1}^kf(\xi_j)$, if for any $\varepsilon>0$, there exists $\sigma>0$ such that if $\Phi\in C^1(\mathcal{M}^k,\mathbb{R})$ with
 \begin{equation*}
   \max\limits_{d_g(\xi_j,\mathcal{C}_j)<\varepsilon,1\leq j\leq k}\big(|F(\bar{\xi})-\Phi(\bar{\xi})|+|\nabla_gF(\bar{\xi})-\nabla_g\Phi(\bar{\xi})|\big)<\delta,
 \end{equation*}
then  $\Phi$ has at least one critical point $\bar{\xi}\in \mathcal{M}^k$ with $d_g(\xi_j,\mathcal{C}_j)<\varepsilon$, where $d_g$ is the geodesic distance on $\mathcal{M}$ with respect to the metric $g$.}
\end{definition}

\begin{remark}\cite[Remark 0.1]{Li}\label{remark}
{\rm $(\mathcal{C}_1,\mathcal{C}_2,\cdots, \mathcal{C}_k)\subset \mathcal{M}^k$ is a
 $C^1$-stable critical set of function $F(\bar{\xi})$ if one of the following conditions holds:

(i) Every $\mathcal{C}_j$ is a strict local minimum (or local maximum) set of $f$, $j=1,2,\cdots,k$.

(ii) Every $\mathcal{C}_j=\{\xi_j^0\}$ is an isolated critical point of $f(\xi_j)$ with $\deg(\nabla_g f,B_g(\xi_j^0,\rho),0)\neq0$ for some $\rho>0$, where $\deg$ is the Brouwer degree, and  $B_g(\xi_j^0,\rho)$ is the ball in $\mathcal{M}$ centered at $\xi_j^0$ with radius $\rho$ with respect to the distance induced by the metric $g$, $j=1,2,\cdots,k$.
}
\end{remark}

Let $L_1,L_2,\cdots,L_7$ be positive numbers defined by
\begin{align}\label{chang}
 \left\{
  \begin{array}{ll}
  \displaystyle L_1=\int\limits_{\mathbb{R}^N}\nabla U_{1,0}(z)\nabla V_{1,0}(z)dz,\\
  \displaystyle L_2=\int\limits_{\mathbb{R}^N}|z|^2\nabla U_{1,0}(z)\nabla V_{1,0}(z)dz,\\
 \displaystyle L_3=\int\limits_{\mathbb{R}^N}U_{1,0}(z)  V_{1,0}(z)dz
    \end{array}
    \right.
    \quad \text{and}\quad\ \
    \left\{
  \begin{array}{ll}
  \displaystyle L_4=\int\limits_{\mathbb{R}^N}|z|^2V_{1,0}^{p+1}(z)dz,\\
   \displaystyle L_5=\int\limits_{\mathbb{R}^N}|z|^2U_{1,0}^{q+1}(z)dz,\\
  \displaystyle L_6=\int\limits_{\mathbb{R}^N}V_{1,0}^{p+1}(z)\log V_{1,0}(z)dz,\\
  \displaystyle  L_7=\int\limits_{\mathbb{R}^N}U_{1,0}^{q+1}(z)\log U_{1,0}(z)dz.
    \end{array}
    \right.
  \end{align}

Our main result states as follows.
\begin{theorem}\label{th}
Let $(\mathcal{M},g)$ be a smooth compact Riemannian manifold, let $h(x)$ be a $C^1$-function on $\mathcal{M}$, 
$(p,q)$ satisfies \eqref{ch}, for any given integer $k\geq2$, set $\bar{\xi^0}=(\xi_1^0,\xi_2^0,\cdots, \xi_k^0)$, let $\xi_j^0$ be an isolated critical point of
\begin{equation}\label{sf}
  \varphi(\xi_j)=h(\xi_j)-\Big(L_2-\frac{L_4}{p+1}-\frac{L_5}{q+1}\Big)\frac{Scal_g(\xi_j)}{6NL_3}
\end{equation}
 with $\varphi(\xi_j^0)>0$ and
$\deg(\nabla_g \varphi,B_g(\xi_j^0,\rho),0)\neq0$ for some $\rho>0$, $j=1,2,\cdots,k$,
 Assume that one of the following conditions holds:

 $(i)$ $\frac{N}{N-2}<p<\frac{N+2}{N-2}$ and $N\geq 8$;

 $(ii)$ $p=\frac{N+2}{N-2}$ and $N\geq10$;

 $(iii)$ $1<p<\frac{N}{N-2}$ and $N\geq 12$.
 \\ Then for $\varepsilon>0$ small enough, system \eqref{pro} admits a family of solutions $(u_\varepsilon,v_\varepsilon)$, which blows up and concentrates at $\bar{\xi}^0$ as $\varepsilon\rightarrow0$.
\end{theorem}

\begin{remark}
{\rm Under the assumptions on $p,q$ and $N$ of Theorem \ref{th}, we have that $L_i<+\infty$ for $i=1,2,\cdots,7$.}
\end{remark}

\begin{remark}
{\rm From the proof of Theorem \ref{th} (see Section \ref{sec3}), it's easy to find that if
\begin{equation*}
  \frac{\alpha}{(p+1)^2}
  +\frac{\beta}{(q+1)^2}>0,
\end{equation*}
then Theorem \ref{th} still holds true. However, in the proof of Proposition \ref{propo1}, we have to impose $\alpha,\beta>0$ to guarantee the continuous embedding, see e.g. \eqref{one}-\eqref{two} and \eqref{three}-\eqref{four}.
}
\end{remark}

\begin{remark}
{\rm If $u=v$, $p=q=\frac{N+2}{N-2}$, $\alpha=\beta=1$, then
\begin{equation*}
  \varphi(\xi_j)=h(\xi_j)-\frac{N-2}{4(N-1)}Scal_g(\xi_j),
\end{equation*}
and Theorem \ref{th} is exactly the conclusion obtained in \cite[Theorem 1.1]{Deng}.
}
\end{remark}

The proof of our result relies on a well known finite dimensional  Lyapunov-Schmidt reduction method, introduced in \cite{BC,FW}. The paper is organized as follows. In Section
\ref{sec2}, we introduce the framework and present some preliminary results. The proof of Theorem \ref{th} is given in Section \ref{sec3}. In Section \ref{sec4}, we perform the finite dimensional reduction, and Section \ref{sec5} is devoted to  the reduced problem. Throughout the paper, $C,C_i$, $i\in \mathbb{N}^+$ denote positive constants possibly different from line to line.

\section{The framework and preliminary results}\label{sec2}

Concerning the least energy solution $(U_{1,0}(z),V_{1,0}(z))$ of \eqref{back} with $\Omega=\mathbb{R}^N$, we have the following asymptotic behaviour and non-degeneracy result.
\begin{lemma}\label{jian1}\cite[Theorem 2]{HV}
Assume that $1<p\leq \frac{N+2}{N-2}$. If $r\rightarrow+\infty$, there hold
\begin{equation*}
 V_{1,0}(r)=O( r^{2-N}),
\end{equation*}
and
\begin{align*}
 U_{1,0}(r) =\left\{
  \begin{array}{ll}
  O( r^{2-N}),\quad &\text{if $p>\frac{N}{N-2}$;}
 \\ O( r^{2-N}\log r),\quad &\text{if $p=\frac{N}{N-2}$;}\\
 O( r^{2-(N-2)p}),\quad &\text{if $p<\frac{N}{N-2}$.}
    \end{array}
    \right.
  \end{align*}
\end{lemma}

\begin{lemma}\label{jian2}\cite[Lemma 2.2]{KM}
Assume that $1<p\leq \frac{N+2}{N-2}$. If $r\rightarrow+\infty$, there hold
\begin{equation*}
 V'_{1,0}(r)=O( r^{1-N}),
\end{equation*}
and
\begin{align*}
 U'_{1,0}(r) =\left\{
  \begin{array}{ll}
  O( r^{1-N}),\quad &\text{if $p>\frac{N}{N-2}$;}
 \\ O( r^{1-N}\log r),\quad &\text{if $p=\frac{N}{N-2}$;}\\
 O( r^{1-(N-2)p}),\quad &\text{if $p<\frac{N}{N-2}$.}
    \end{array}
    \right.
  \end{align*}
\end{lemma}

\begin{lemma}\label{jian3}\cite[Remark 2.3]{GP}
Assume that $1<p\leq \frac{N+2}{N-2}$. If $r\rightarrow+\infty$, there hold
\begin{equation*}
 V''_{1,0}(r)=O( r^{-N}),
\end{equation*}
and
\begin{align*}
U''_{1,0}(r) = \left\{
  \begin{array}{ll}
  O( r^{-N}),\quad &\text{if $p>\frac{N}{N-2}$;}
 \\ O( r^{-N}\log r),\quad &\text{if $p=\frac{N}{N-2}$;}\\
 O( r^{-(N-2)p}),\quad &\text{if $p<\frac{N}{N-2}$.}
    \end{array}
    \right.
  \end{align*}
\end{lemma}

\begin{lemma}\label{nonde}\cite[Theorem 1]{FKP}
Set
\begin{equation*}
  (\Psi_{1,0}^1,\Phi_{1,0}^1)=\Big(x\cdot \nabla U_{1,0}+\frac{N U_{1,0}}{q+1},x\cdot \nabla V_{1,0}+\frac{N V_{1,0}}{p+1}\Big)
\end{equation*}
and
\begin{equation*}
  (\Psi_{1,0}^l,\Phi_{1,0}^l)=\big(\partial _l U_{1,0},\partial _l V_{1,0} \big),\quad \text{for $l=1,2,\cdots,N$.}
\end{equation*}
Then the space of solutions to the linear system
\begin{align*}
 \left\{
  \begin{array}{ll}
  -\Delta\Psi=pV_{1,0}^{p-1}\Phi,
  \ \  \mbox{in}\ \mathbb{R}^N,\\
   -\Delta\Phi=qU_{1,0}^{q-1}\Psi,
  \ \  \mbox{in}\ \mathbb{R}^N,\\
  (\Psi,\Phi)\in \dot{W}^{2,\frac{p+1}{p}}(\mathbb{R}^N)\times \dot{W}^{2,\frac{q+1}{q}}(\mathbb{R}^N)
    \end{array}
    \right.
  \end{align*}
  is spanned by
  \begin{equation*}
     \big\{(\Psi_{1,0}^0,\Phi_{1,0}^0), (\Psi_{1,0}^1,\Phi_{1,0}^1),\cdots, (\Psi_{1,0}^N,\Phi_{1,0}^N)\big\}.
  \end{equation*}
\end{lemma}

Moreover, we have the following elementary inequality.
\begin{lemma}\cite[Lemma 2.1]{LN}\label{gs}
For any $a>0$, $b$ real, there holds
\begin{align*}
 \big||a+b|^\beta-b^\beta\big|\leq\left\{
  \begin{array}{ll}
  C(\beta)(a^{\beta-1}|b|+|b|^\beta),\quad &\text{if $\beta\geq1$},\\
  C(\beta)\min\big\{a^{\beta-1}|b|,|b|^\beta\big\},\quad &\text{if $0<\beta<1$}.
    \end{array}
    \right.
  \end{align*}
\end{lemma}
Now, we recall some definitions and results about the compact Riemannian manifold $(\mathcal{M},g)$.
\begin{definition}
{\rm Let $(\mathcal{M},g)$ be a smooth compact Riemannian manifold. On the tangent bundle of $\mathcal{M}$, define the exponential map $\exp: T \mathcal{M}\rightarrow \mathcal{M}$, which has the following properties:

(i) $\exp$ is of  class $C^\infty$;

(ii) there exists a constant $r_0>0$ such that $\exp_\xi|_{B(0,r_0)}\rightarrow B_g(\xi,r_0)$ is a diffeomorphism for all $\xi\in \mathcal{M}$.
}
\end{definition}

Fix such $r_0$ in this paper with $r_0<i_g/{2}$, where $i_g$ denotes the injectivity radius of $(\mathcal{M},g)$. For
 any $1<s<+\infty$ and $u\in L^s(\mathcal{M})$, we denote the $L^s$-norm of $u$ by $$ \|u\|_s=\Big(\int\limits_{\mathcal{M}}|u|^sd v_g\Big)^{1/{s}},$$
where $d v_g=\sqrt{\det g}dz$ is the volume element on $\mathcal{M}$ associated to the metric $g$.
We
introduce the Banach space
\begin{equation*}
  \mathcal{X}_{p,q}(\mathcal{M})=\dot{W}^{2,\frac{p+1}{p}}(\mathcal{M})\times \dot{W}^{2,\frac{q+1}{q}}(\mathcal{M})
\end{equation*}
equipped with the norm
\begin{equation}\label{norm}
  \|(u,v)\|=\|\Delta_g u\|_{{\frac{p+1}{p}}}+\|\Delta_g v\|_{{\frac{q+1}{q}}}.
\end{equation}
Denote by $\mathcal{I}^*$ the formal adjoint operator of the embedding $\mathcal{I}:\mathcal{X}_{q,p}(\mathcal{M})\hookrightarrow L^{p+1}(\mathcal{M})\times L^{q+1}(\mathcal{M})$. By the Calder\'{o}n-Zygmund estimate, the operator $\mathcal{I}^*$ maps $L^{\frac{p+1}{p}}(\mathcal{M})\times L^{\frac{q+1}{q}}(\mathcal{M})$ to $\mathcal{X}_{p,q}(\mathcal{M})$. Then we rewrite problem \eqref{pro} as
\begin{equation}\label{repro}
  (u,v)=\mathcal{I}^*\big(f_\varepsilon(v),g_\varepsilon(u)\big).
\end{equation}
where $f_\varepsilon(u):=u_+^{p-\alpha\varepsilon}$, $g_\varepsilon(u):=u_+^{q-\beta\varepsilon}$ and $u_+=\max \{u,0\}$.
Moreover, by the Sobolev embedding theorem, we have
\begin{equation}\label{em}
  \|\mathcal{I}^*\big(f_\varepsilon(v),g_\varepsilon(u)\big)\|\leq C\|f_\varepsilon(v)\|_{{\frac{p+1}{p}}}+
  C\|g_\varepsilon(u)\|_{{\frac{q+1}{q}}},
\end{equation}
and
\begin{equation}\label{embedd}
  \mathcal{X}_{p,q}(\mathcal{M})\hookrightarrow \dot{W}^{1,p^*}(\mathcal{M})\times \dot{W}^{1,q^*}(\mathcal{M}),\quad \mathcal{X}_{p,q}(\mathcal{M})\hookrightarrow L^2(\mathcal{M})\times L^2(\mathcal{M}).
\end{equation}

Let $\chi$ be a smooth cutoff function such that $0\leq \chi\leq1$ in $\mathbb{R}^+$, $\chi=1$ in $[0,r_0/{2}]$, and $\chi=0$ out of $[r_0,+\infty]$. For any $\xi\in \mathcal{M}$ and $\delta>0$, we define the following functions on $\mathcal{M}$
\begin{equation}\label{fun1}
  (W_{\delta,\xi}(x),H_{\delta,\xi}(x)):=\big(\chi(d_g(x,\xi))\delta^{-\frac{N}{q+1}}U_{1,0}(\delta^{-1}\exp_\xi^{-1}(x)), \chi(d_g(x,\xi))\delta^{-\frac{N}{p+1}}V_{1,0}(\delta^{-1}\exp_\xi^{-1}(x))\big)
\end{equation}
and
\begin{equation*}
  (\Psi^i_{\delta,\xi}(x),\Phi^i_{\delta,\xi}(x)):=\big(\chi(d_g(x,\xi))\delta^{-\frac{N}{q+1}}\Psi_{1,0}^i(\delta^{-1}\exp_\xi^{-1}(x)), \chi(d_g(x,\xi))\delta^{-\frac{N}{p+1}}\Phi_{1,0}^i(\delta^{-1}\exp_\xi^{-1}(x))\big),
\end{equation*}
for $i=0,1,\cdots, N$,
where $(\Psi_{1,0}^i,\Phi_{1,0}^i)$ is given in Lemma \ref{nonde}.

For any $\varepsilon>0$ and $\bar{t}=(t_1,t_2,\cdots,t_k)\in (\mathbb{R}^+)^k$, we set
\begin{equation}\label{solu'}
  \bar{\delta}=(\delta_1,\delta_2,\cdots,\delta_k)\in (\mathbb{R}^+)^k,\quad \delta_j=\sqrt{\varepsilon t_j},\quad \varrho_1<t_j<\frac{1}{\varrho_1},
\end{equation}
for fixed small $\varrho_1>0$.
Moreover, for $\varrho_2\in (0,1)$, we define the configuration space
$\Lambda$ by
\begin{align*}
  \Lambda=&\Big\{(\bar{\delta},\bar{\xi}):\bar{\delta}=(\delta_1,\delta_2,\cdots,\delta_k)\in (\mathbb{R}^+)^k,\,\, \bar{\xi}=(\xi_1,\xi_2,\cdots,\xi_k)\in \mathcal{M}^k,\\
  &\ \ \ d_g(\xi_j,\xi_m)\geq \varrho_2>2r_0 \,\,\text{for $j,m=1,2,\cdots,k$ and $j\neq m$}\Big\}.
\end{align*}
Let $\mathcal{Y}_{\bar{\delta},\bar{\xi}}$ and $\mathcal{Z}_{\bar{\delta},\bar{\xi}}$ be two subspaces of $\mathcal{X}_{p,q}(\mathcal{M})$ given as
\begin{equation*}
  \mathcal{Y}_{\bar{\delta},\bar{\xi}}=span\big\{(\Psi^i_{\delta_j,\xi_j},\Phi^i_{\delta_j,\xi_j}):i=0,1,\cdots, N \,\, \text{and}\,\, j=1,2,\cdots,k\big\}
\end{equation*}
and
\begin{equation*}
  \mathcal{Z}_{\bar{\delta},\bar{\xi}}=\big\{(\Psi,\Phi)\in \mathcal{X}_{p,q}(\mathcal{M}): \langle(\Psi,\Phi),(\Psi^i_{\delta_j,\xi_j},\Phi^i_{\delta_j,\xi_j})\rangle_h=0 \,\, \text{ for $i=0,1,\cdots, N$ and $j=1,2,\cdots,k$} \big\},
\end{equation*}
where
\begin{equation*}
  \langle(u,v),(\varphi,\psi)\rangle_h=\int\limits_{\mathcal{M}}(\nabla _g u \cdot \nabla _g \psi+\nabla _g v \cdot \nabla _g \varphi )d v_g+\int\limits_{\mathcal{M}}(hu \psi +hv \varphi)dv_g
\end{equation*}
for any $(u,v),(\varphi,\psi)\in \mathcal{X}_{p,q}(\mathcal{M})$.

\begin{lemma}\label{topo}
There exists $\varepsilon_0>0$ such that for any $\varepsilon\in (0,\varepsilon_0)$, $\mathcal{X}_{p,q}(\mathcal{M})=\mathcal{Y}_{\bar{\delta},\bar{\xi}}\oplus \mathcal{Z}_{\bar{\delta},\bar{\xi}}$.
\end{lemma}
\begin{proof}
We shall prove that for any $(\Psi,\Phi)\in \mathcal{X}_{p,q}(\mathcal{M})$, there exists unique pair $(\Psi_0,\Phi_0)\in \mathcal{Z}_{\bar{\delta},\bar{\xi}}$ and coefficients $c_{10},c_{11},\cdots,c_{1N}, c_{20},c_{21},\cdots,c_{2N},\cdots,c_{k0},c_{k1},\cdots,c_{kN}$ such that
\begin{equation}\label{coe}
  (\Psi,\Phi)=(\Psi_0,\Phi_0)+\sum\limits_{l=0}^N\sum\limits_{m=1}^kc_{lm}(\Psi^l_{\delta_m,\xi_m},\Phi^l_{\delta_m,\xi_m}).
\end{equation}
The requirement that $(\Psi_0,\Phi_0)\in \mathcal{Z}_{\bar{\delta},\bar{\xi}}$ is equivalent to demanding
\begin{align}\label{dem}
  &\int\limits_{\mathcal{M}}\big(\nabla _g \Psi \cdot\nabla _g \Phi^i_{\delta_j,\xi_j}+\nabla _g \Phi\cdot\nabla _g \Psi^i_{\delta_j,\xi_j}+h\Psi \Phi^i_{\delta_j,\xi_j}+h\Phi \Psi^i_{\delta_j,\xi_j}\big)d v_g \nonumber\\
  =&\sum\limits_{l=0}^N\sum\limits_{m=1}^kc_{lm}\int\limits_{\mathcal{M}}\big (\nabla _g \Psi^l_{\delta_m,\xi_m} \cdot\nabla _g \Phi^i_{\delta_j,\xi_j}+\nabla _g \Phi^l_{\delta_m,\xi_m} \cdot\nabla _g \Psi^i_{\delta_j,\xi_j}+h\Psi^l_{\delta_m,\xi_m} \Phi^i_{\delta_j,\xi_j}+h\Phi^l_{\delta_m,\xi_m}  \Psi^i_{\delta_j,\xi_j}\big) d v_g
\end{align}
for any $i=0,1,\cdots,N$ and $j=1,2,\cdots,k$.

We estimate the integral on the right-hand side of \eqref{dem}. By standard properties of the exponential map, there exists $C>0$ such that for any $\xi\in \mathcal{M}$, $\delta>0$, $z\in B(0,r_0/\delta)$, and $i,j,k\in \mathbb{N}^+$, there hold
\begin{equation}\label{lap1}
  |g_{\delta,\xi}^{ij}(z)-Eucl^{ij}|\leq C\delta^2|z|^2,\quad \text{and}\quad |g_{\delta,\xi}^{ij}(z) (\Gamma_{\delta,\xi})_{ij}^k(z)|\leq C\delta^2|z|,
\end{equation}
where $g_{\delta,\xi}(z)=\exp_\xi^*g(\delta z)$ and $(\Gamma_{\delta,\xi})_{ij}^k$ stand for the Christoffel symbols of the metric $g_{\delta,\xi}$.
Taking into account that there holds
\begin{equation}\label{lap2}
  \Delta _{g_{\delta,\xi}}=g_{\delta,\xi}^{ij}\Big(\frac{\partial ^2}{\partial z_i \partial z_j}-(\Gamma_{\delta,\xi})_{ij}^k\frac{\partial }{\partial z_k}\Big),
\end{equation}
by Lemma \ref{nonde} and $dg(\xi_j,\xi_m)>2r_0$ for $j\neq m$,  we have
\begin{align}\label{gu1}
  &\int\limits_{\mathcal{M}}\nabla _g \Psi^l_{\delta_m,\xi_m} \cdot\nabla _g \Phi^i_{\delta_j,\xi_j}d v_g=\delta_{jm}\int\limits_{\mathcal{M}}\nabla _g \Psi^l_{\delta_j,\xi_j} \cdot\nabla _g \Phi^i_{\delta_j,\xi_j}d v_g\nonumber\\
  =&\delta_{jm}\int\limits_{B(0,r_0/\delta_j)}\nabla _{g_{\delta_j,\xi_j}} (\chi_{\delta_j}\Psi^l_{1,0})\cdot  \nabla _{g_{\delta_j,\xi_j}} (\chi_{\delta_j}\Phi^i_{1,0})dz \nonumber\\
  =&p\delta_{jm}\int\limits_{B(0,r_0/{\delta_j})}\chi^2_{\delta_j} V_{1,0}^{p-1}\Phi_{1,0}^l\Phi_{1,0}^idz+O(\delta_j^2) \nonumber \\=&p\delta_{il}\delta_{jm}\int\limits_{B(0,r_0/{\delta_j})}\chi^2_{\delta_j} V_{1,0}^{p-1}(\Phi_{1,0}^i)^2dz+O(\delta_j^2),
\end{align}
and
\begin{align}\label{gu2}
  &\int\limits_{\mathcal{M}}h\Psi^l_{\delta_m,\xi_m} \Phi^i_{\delta_j,\xi_j}d v_g=\delta_{jm}\int\limits_{\mathcal{M}}h\Psi^l_{\delta_j,\xi_j} \Phi^i_{\delta_j,\xi_j}d v_g
  =\delta_{jm}\delta_j^2 \int\limits_{B(0,r_0/\delta_j)}h_{\delta_j,\xi_j} \Psi_{1,0}^l\Phi_{1,0}^idz \nonumber\\ =&-\delta_{jm}\delta_j^2 \int\limits_{B(0,r_0/\delta_j)}h_{\delta_j,\xi_j}\frac{\Delta \Phi_{1,0}^l}{qU_{1,0}^{q-1}}\Phi_{1,0}^idz+o(\delta_j^2)=\delta_{il}  \delta_{jm}  \delta_j^2 \int\limits_{B(0,r_0/\delta_j)}h_{\delta_j,\xi_j} \frac{(\nabla\Phi_{1,0}^i)^2}{qU_{1,0}^{q-1}}dz+o(\delta_j^2),
\end{align}
where
 $\chi_{\delta_j}(x)=\chi({\delta_j|z|})$ and $h_{\delta_j,\xi_j}(z)=h(\exp_{\xi_j}(\delta_j z))$.
Similarly, we have
\begin{align}\label{gu3}
  \int\limits_{\mathcal{M}}\nabla _g \Phi^l_{\delta_m,\xi_m} \cdot\nabla _g \Psi^i_{\delta_j,\xi_j}d v_g=q\delta_{il}\delta_{jm}\int\limits_{B(0,r_0/{\delta_j})}\chi^2_{\delta_j} U_{1,0}^{q-1}(\Psi_{1,0}^i)^2dz+O(\delta_j^2),
\end{align}
and
\begin{align}\label{gu4}
  \int\limits_{\mathcal{M}}h\Phi^l_{\delta_m,\xi_m} \Psi^i_{\delta_j,\xi_j}d v_g=\delta_{il}\delta_{jm}\delta_j^2  \int\limits_{B(0,r_0/\delta_j)}h_{\delta_j,\xi_j} \frac{(\nabla\Psi_{1,0}^i)^2}{pV_{1,0}^{p-1}}dz+o(\delta_j^2).
\end{align}
By plugging \eqref{gu1}-\eqref{gu4} into \eqref{dem}, we can see that the coefficients $c_{lm}$ are uniquely determined for $l=0,1,\cdots,N$ and $m=1,2,\cdots,k$. By virtue of \eqref{coe}, so is $(\Psi_0,\Phi_0)$.

On the other hand, $\mathcal{Y}_{\bar{\delta},\bar{\xi}}$ and $\mathcal{Z}_{\bar{\delta},\bar{\xi}}$ are clearly closed subspaces of $\mathcal{X}_{p,q}(\mathcal{M})$, Therefore, they are topological complements of each other.
\end{proof}

\section{Scheme of the proof of Theorem \ref{th}}\label{sec3}

We look for solutions of system \eqref{pro}, or equivalently of  \eqref{repro}, of the form
\begin{equation}\label{solu}
  (u_\varepsilon,v_\varepsilon)=\big(\mathcal{W}_{\bar{\delta},\bar{\xi}}+\Psi_{\varepsilon,\bar{t},\bar{\xi}}
  ,\mathcal{H}_{\bar{\delta},\bar{\xi}}+\Phi_{\varepsilon,\bar{t},\bar{\xi}}\big), \quad\mathcal{W}_{\bar{\delta},\bar{\xi}}=\sum\limits_{j=0}^kW_{\delta_j,\xi_j}, \quad \mathcal{H}_{\bar{\delta},\bar{\xi}}=\sum\limits_{j=0}^kH_{\delta_j,\xi_j},\quad (\bar{\delta},\bar{\xi})\in \Lambda,
\end{equation}
where $\bar{\delta}$ is as in \eqref{solu'}, $(W_{\delta_j,\xi_j},H_{\delta_j,\xi_j})$ is as in \eqref{fun1}, and $(\Psi_{\varepsilon,\bar{t},\bar{\xi}},\Phi_{\varepsilon,\bar{t},\bar{\xi}})\in \mathcal{Z}_{\bar{\delta},\bar{\xi}}$. By Lemma \ref{topo}, $\mathcal{X}_{p,q}(\mathcal{M})=\mathcal{Y}_{\bar{\delta},\bar{\xi}}\oplus \mathcal{Z}_{\bar{\delta},\bar{\xi}}$. Then we define the projections $\Pi_{\bar{\delta},\bar{\xi}}$ and $\Pi_{\bar{\delta},\bar{\xi}}^\bot$ of the Sobolev space $\mathcal{X}_{p,q}(\mathcal{M})$ onto $\mathcal{Y}_{\bar{\delta},\bar{\xi}}$ and $\mathcal{Z}_{\bar{\delta},\bar{\xi}}$ respectively. Therefore, we have to solve the couples of equations
\begin{equation}\label{tou1}
  \Pi_{\bar{\delta},\bar{\xi}}\Big[\big(\mathcal{W}_{\bar{\delta},\bar{\xi}}+\Psi_{\varepsilon,\bar{t},\bar{\xi}}
  ,\mathcal{H}_{\bar{\delta},\bar{\xi}}+\Phi_{\varepsilon,\bar{t},\bar{\xi}}\big)
  -\mathcal{I}^*\big(f_\varepsilon(\mathcal{H}_{\bar{\delta},\bar{\xi}}+\Phi_{\varepsilon,\bar{t},\bar{\xi}}),g_\varepsilon(\mathcal{W}_{\bar{\delta},\bar{\xi}}+\Psi_{\varepsilon,\bar{t},\bar{\xi}})\big)\Big]=0,
\end{equation}
and
\begin{equation}\label{tou2}
  \Pi_{\bar{\delta},\bar{\xi}}^\bot\Big[\big(\mathcal{W}_{\bar{\delta},\bar{\xi}}+\Psi_{\varepsilon,\bar{t},\bar{\xi}}
  ,\mathcal{H}_{\bar{\delta},\bar{\xi}}+\Phi_{\varepsilon,\bar{t},\bar{\xi}}\big)
  -\mathcal{I}^*\big(f_\varepsilon(\mathcal{H}_{\bar{\delta},\bar{\xi}}+\Phi_{\varepsilon,\bar{t},\bar{\xi}}),g_\varepsilon(\mathcal{W}_{\bar{\delta},\bar{\xi}}+\Psi_{\varepsilon,\bar{t},\bar{\xi}})\big)\Big]=0.
\end{equation}

The first step in the proof consists in solving equation \eqref{tou2}. This requires Proposition \ref{propo1} below, whose proof is postponed to Section \ref{sec4}.
\begin{proposition}\label{propo1}
Under the assumptions of Theorem \ref{th}, if $(\bar{\delta},\bar{\xi})\in \Lambda$ and $\bar{\delta}$ is as in \eqref{solu'},
then for any $\varepsilon>0$ small enough,
equation \eqref{tou2} admits a unique solution $(\Psi_{\varepsilon,\bar{t},\bar{\xi}},\Phi_{\varepsilon,\bar{t},\bar{\xi}})$ in $\mathcal{Z}_{\bar{\delta},\bar{\xi}}$, which is continuously differentiable with respect to $\bar{t}$ and $\bar{\xi}$, such that
\begin{equation*}
  \|(\Psi_{\varepsilon,\bar{t},\bar{\xi}},\Phi_{\varepsilon,\bar{t},\bar{\xi}})\|\leq C\varepsilon|\log \varepsilon|.
\end{equation*}
\end{proposition}
We now introduce the energy functional $\mathcal{J}_\varepsilon$ defined on $\mathcal{X}_{p,q}(\mathcal{M})$ by
\begin{equation*}
  \mathcal{J}_\varepsilon(u,v)=\int\limits_{\mathcal{M}}\nabla _g u\cdot \nabla _g vd v_g+\int\limits_{\mathcal{M}}huvd v_g-\frac{1}{p+1-\alpha\varepsilon}\int\limits_{\mathcal{M}}
  v^{p+1-\alpha\varepsilon}d v_g-\frac{1}{q+1-\beta\varepsilon}\int\limits_{\mathcal{M}}
  u^{q+1-\beta\varepsilon}d v_g.
\end{equation*}
It is clear that the critical points of $\mathcal{J}_\varepsilon$ are the solutions of system \eqref{pro}. Moreover,
\begin{align*}
  \mathcal{J}'_\varepsilon(u,v)(\varphi,\psi)=&\int\limits_{\mathcal{M}}(\nabla _g u\cdot \nabla _g \psi+\nabla _g v\cdot \nabla _g \varphi )d v_g+\int\limits_{\mathcal{M}}(hu\psi+hv\varphi)d v_g-\int\limits_{\mathcal{M}}
  u^{q-\beta\varepsilon} \varphi d v_g-\int\limits_{\mathcal{M}}
  v^{p-\alpha\varepsilon} \psi d v_g,
\end{align*}
for any $(u,v),(\varphi,\psi)\in \mathcal{X}_{p,q}(\mathcal{M})$.
We also define the functional $\widetilde{\mathcal{J}}_\varepsilon:(\mathbb{R}^+)^k\times \mathcal{M}^k\rightarrow\mathbb{R}$
\begin{equation}\label{defj}
  \widetilde{\mathcal{J}}_\varepsilon(\bar{t},\bar{\xi})=\mathcal{J}_\varepsilon\big(\mathcal{W}_{\bar{\delta},\bar{\xi}}+\Psi_{\varepsilon,\bar{t},\bar{\xi}}
  ,\mathcal{H}_{\bar{\delta},\bar{\xi}}+\Phi_{\varepsilon,\bar{t},\bar{\xi}}\big),
\end{equation}
where $(\mathcal{W}_{\bar{\delta},\bar{\xi}},\mathcal{H}_{\bar{\delta},\bar{\xi}})$ is as \eqref{solu}, $(\Psi_{\varepsilon,\bar{t},\bar{\xi}},\Phi_{\varepsilon,\bar{t},\bar{\xi}})$ is given in Proposition \ref{propo1}.
\begin{definition}
{\rm For a given $C^1$-function $\varphi_\varepsilon$, we say that the estimate $\varphi_\varepsilon=o(\varepsilon)$ is $C^1$-uniform if there hold $\varphi_\varepsilon=o(\varepsilon)$ and $\nabla \varphi_\varepsilon=o(\varepsilon)$ as $\varepsilon\rightarrow0$.}
\end{definition}
We solve equation \eqref{tou1} in Proposition \ref{propo2} below whose proof is postponed to Section \ref{sec5}.
\begin{proposition}\label{propo2}
(i) Under the assumptions of Theorem \ref{th}, if $\bar{\delta}$ is as in \eqref{solu'}, for any $\varepsilon>0$ small enough, if $(\bar{t},\bar{\xi})$ is a critical point of the functional $\widetilde{\mathcal{J}}_\varepsilon$, then $\big(\mathcal{W}_{\bar{\delta},\bar{\xi}}+\Psi_{\varepsilon,\bar{t},\bar{\xi}}
  ,\mathcal{H}_{\bar{\delta},\bar{\xi}}+\Phi_{\varepsilon,\bar{t},\bar{\xi}}\big)$ is a  solution of system \eqref{pro}, or equivalently of  \eqref{repro}.

(ii) Under the assumptions of Theorem \ref{th},
there holds
\begin{align*}
  \widetilde{\mathcal{J}}_{\varepsilon}(\mathcal{W}_{\bar{\delta},\bar{\xi}},\mathcal{H}_{\bar{\delta},\bar{\xi}})=
  &\frac{2k}{N}L_1+c_1\varepsilon-c_2\varepsilon  \log \varepsilon +\Psi_k(\bar{t},\bar{\xi})\varepsilon +o(\varepsilon),
\end{align*}
as $\varepsilon\rightarrow0$, $C^1$-uniformly with respect to $\bar{\xi}$ in $\mathcal{M}^k$ and to $\bar{t}$ in compact subsets of $(\mathbb{R}^+)^k$, where
\begin{equation}\label{deffa}
  \Psi_k(\bar{t},\bar{\xi})=\sum\limits_{j=1}^k\Big\{L_3\varphi(\xi_j)t_j-\frac{NL_1}{2}\Big(\frac{\alpha}{(p+1)^2}
  +\frac{\beta}{(q+1)^2}\Big)  \log (t_j)\Big\},
\end{equation}
and
\begin{equation}\label{defc1c2}
c_1=\Big[\Big(\frac{L_6\alpha}{p+1}+\frac{L_7\beta}{q+1}\Big)
  -\Big(\frac{\alpha}{(p+1)^2}+\frac{\beta}{(q+1)^2}\Big)
  L_1\Big]k,\ \ \ c_2=\frac{NL_1k}{2}\Big(\frac{\alpha}{(p+1)^2}
  +\frac{\beta}{(q+1)^2}\Big),
\end{equation}
with $L_1,L_3,L_6,L_7$ are positive constants given in \eqref{chang}, $\varphi(\xi_j)$ is defined as \eqref{sf}, $j=1,2,\cdots,k$.
\end{proposition}

We now prove Theorem \ref{th} by using Propositions \ref{propo1} and \ref{propo2}.\\
{\bf Proof of Theorem \ref{th}.} Define $\widetilde{\mathcal{J}}:(\mathbb{R}^+)^k\times \mathcal{M}^k\rightarrow \mathbb{R}$ by
\begin{equation*}
  \widetilde{\mathcal{J}}(\bar{t},\bar{\xi})=\sum\limits_{j=1}^kf(t_j,\xi_j),\quad \text{with $f(t_j,\xi_j)=-\widetilde{C}\log t_j+L_3\varphi(\xi_j)t_j$},
\end{equation*}
where $\widetilde{C}=\big(\frac{\alpha}{(p+1)^2}
  +\frac{\beta}{(q+1)^2}\big)\frac{NL_1}{2}$ and $L_1,L_3>0$ are given in \eqref{chang}.
Since $\xi_j^0$ is an isolated critical  point of the function $\varphi(\xi_j)$ with $\varphi(\xi_j^0)>0$, and set $t_j^0=\frac{\widetilde{C}}{L_3\varphi(\xi_j^0)}$, then $t_j^0>0$ and $(t_j^0,\xi_j^0)$ is an isolated critical point of $f(t_j,\xi_j)$. Moreover, by $\deg(\nabla_g \varphi,B_g(\xi_j^0,\rho),0)\neq0$ for some $\rho>0$, we obtain $\deg(\nabla_g f,B_g(\xi_j^0,\rho),0)\neq0$, $j=1,2,\cdots,k$.
Hence, by Remark \ref{remark}, $(\bar{t^0},\bar{\xi^0})$ is a $C^1$-stable critical set of $\widetilde{\mathcal{J}}$, where $\bar{t^0}=(t_1^0,t_2^0,\cdots,t_k^0)$ and $\bar{\xi^0}=(\xi_1^0,\xi_2^0,\cdots,\xi_k^0)$.
Using Proposition \ref{propo2}, we have
\begin{equation*}
  \Big|\partial _{\bar{t}}\big(\varepsilon^{-1}\widetilde{\mathcal{J}}_\varepsilon-\widetilde{\mathcal{J}}\big)\Big|+\Big|\partial_{\bar{\xi}} \big(\varepsilon^{-1}\widetilde{\mathcal{J}}_\varepsilon-\widetilde{\mathcal{J}}\big)\Big|\rightarrow0,
\end{equation*}
as $\varepsilon\rightarrow0$, uniformly with respect to $\bar{\xi}$ in $\mathcal{M}^k$ and to $\bar{t}$ in compact subsets of $(\mathbb{R}^+)^k$. By standard properties of the Brouwer degree, it follows that there exists a family of critical points $(\bar{t^\varepsilon},\bar{\xi^\varepsilon})$ of $\widetilde{\mathcal{J}}_\varepsilon$ converging to $(\bar{t^0},\bar{\xi^0})$ as $\varepsilon\rightarrow0$. Using Proposition \ref{propo2} again, we can see that the function $(u_\varepsilon,v_\varepsilon)=\big(\mathcal{W}_{\bar{\delta^\varepsilon},\bar{\xi^\varepsilon}}+\Psi_{\varepsilon,\bar{t^\varepsilon},\bar{\xi^\varepsilon}}
  ,\mathcal{H}_{\bar{\delta^\varepsilon},\bar{\xi^\varepsilon}}+\Phi_{\varepsilon,\bar{t^\varepsilon},\bar{\xi^\varepsilon}}\big)$ is a pair of solutions of system \eqref{pro} for any $\varepsilon>0$ small enough, where $\bar{\delta^\varepsilon}$ is as in \eqref{solu'}. Moreover,
$(u_\varepsilon,v_\varepsilon)$ blows up and concentrates at $\bar{\xi^0}$ at $\varepsilon\rightarrow0$. This ends the proof. \qed

\section{Proof of Proposition \ref{propo1}}\label{sec4}
This section is devoted to the proof of Proposition \ref{propo1}. For any $\varepsilon>0$, $\bar{t}\in (\mathbb{R}^+)^k$, and $\bar{\xi}\in \mathcal{M}^k$, if $\bar{\delta}$ is as in \eqref{solu'}, we introduce the map
$\mathcal{L}_{\varepsilon,\bar{t},\bar{\xi}}:\mathcal{Z}_{\bar{\delta},\bar{\xi}}\rightarrow \mathcal{Z}_{\bar{\delta},\bar{\xi}}$ defined by
\begin{equation}\label{defl}
\mathcal{L}_{\varepsilon,\bar{t},\bar{\xi}}(\Psi,\Phi)=
  \Pi_{\bar{\delta},\bar{\xi}}^\bot\Big[
  (\Psi,\Phi)-\mathcal{I}^*\big(f'_\varepsilon(\mathcal{H}_{\bar{\delta},\bar{\xi}})\Phi,g'_\varepsilon(\mathcal{W}_{\bar{\delta},\bar{\xi}})\Psi\big)
  \Big].
\end{equation}
It's easy to check that $\mathcal{L}_{\varepsilon,\bar{t},\bar{\xi}}$ is well defined in $\mathcal{Z}_{\bar{\delta},\bar{\xi}}$. Next, we prove the invertibility of this map.
\begin{lemma}\label{line}
Under the assumptions on $p,q$ and $N$ of Theorem \ref{th}, if $(\bar{\delta},\bar{\xi})\in \Lambda$ and $\bar{\delta}$ is as in \eqref{solu'}, then for any $\varepsilon>0$ small enough, and $(\Psi,\Phi)\in \mathcal{Z}_{\bar{\delta},\bar{\xi}}$, there holds
\begin{equation*}
  \|\mathcal{L}_{\varepsilon,\bar{t},\bar{\xi}}(\Psi,\Phi)\|\geq C\|(\Psi,\Phi)\|,
\end{equation*}
where $\mathcal{L}_{\varepsilon,\bar{t},\xi}(\Psi,\Phi)$ is as in \eqref{defl}.
\end{lemma}
\begin{proof}
We assume by contradiction that there exist a sequence $\varepsilon_n\rightarrow 0$ as $n\rightarrow+\infty$, $(\bar{\delta_n},\bar{\xi_n})\in \Lambda$, $\bar{t_n}=(t_{1n},t_{2n},\cdots,t_{kn})\in (\mathbb{R}^+)^k$, $\bar{\xi_n}=(\xi_{1n},\xi_{2n},\cdots,\xi_{kn})\in \mathcal{M}^k$, and a sequence of functions $(\Psi_n,\Phi_n)\in\mathcal{Z}_{\bar{\delta_n},\bar{\xi_n}}$ such that
\begin{equation*}
  \|(\Psi_n,\Phi_n)\|=1,\quad \|\mathcal{L}_{\varepsilon_n,\bar{t_n},\bar{\xi_n}}(\Psi_n,\Phi_n)\|\rightarrow0, \quad \text{as $n\rightarrow+\infty$}.
\end{equation*}
Then $\|\Psi_n\|_{{q+1}}\leq C$ and $\|\Phi_n\|_{{p+1}}\leq C$.

{\bf Step 1:}
For any $n\in \mathbb{N}^+$ and $j=1,2,\cdots,k$, let
\begin{equation*}
  (\widetilde{\Psi}_n(x),\widetilde{\Phi}_n(x))=\big(\chi(\delta_{jn}|x|)\delta_{jn}^{\frac{N}{q+1}}\Psi_n(\exp _{\xi_{jn}}(\delta_{jn} x)),\chi(\delta_{jn}|x|)\delta_{jn}^{\frac{N}{p+1}}\Phi_n(\exp _{\xi_{jn}}(\delta_{jn} x))\big),
\end{equation*}
where $\chi$ is a cutoff function as in \eqref{fun1}. A direct computations shows 
  \begin{align*}
  \big\|\Delta \widetilde{\Psi}_n\big\|^\frac{p+1}{p}_{L^{\frac{p+1}{p}}(\mathbb{R}^N)}&\leq \int\limits_{B(0,r_0/\delta_{jn})}| \delta_{jn}^{\frac{N}{q+1}}\Delta\Psi_n(\exp _{\xi_{jn}}(\delta_{jn} x))|^{\frac{p+1}{p}}dx=\int\limits_{B(0,r_0)}\delta_{jn}^{-N}|\delta_{jn}^{2+\frac{N}{q+1}}\Delta\Psi_n(\exp _{\xi_{jn}}(y))|^{\frac{p+1}{p}}dy\\
  &=\int\limits_{B_g(\xi_{jn},r_0)}|\Delta_g\Psi_n|^{\frac{p+1}{p}}d v_g=\int\limits_{\mathcal{M}}|\Delta_g\Psi_n|^{\frac{p+1}{p}}d v_g\leq C,
\end{align*}
and
\begin{align*}
  \big\|\Delta \widetilde{\Phi}_n\big\|^\frac{q+1}{q}_{L^{\frac{p+1}{p}}(\mathbb{R}^N)}&\leq \int\limits_{B(0,r_0/\delta_{jn})}| \delta_{jn}^{\frac{N}{p+1}}\Delta\Phi_n(\exp _{\xi_{jn}}(\delta_{jn} x))|^{\frac{q+1}{q}}dx=\int\limits_{B(0,r_0)}\delta_{jn}^{-N}|\delta_{jn}^{2+\frac{N}{p+1}}\Delta\Phi_n(\exp _{\xi_{jn}}(y))|^{\frac{q+1}{q}}dy\\
  &=\int\limits_{B_g(\xi_{jn},r_0)}|\Delta_g\Phi_n|^{\frac{q+1}{q}}d v_g=\int\limits_{\mathcal{M}}|\Delta_g\Phi_n|^{\frac{q+1}{q}}d v_g\leq C.
\end{align*}
Hence, $(\widetilde{\Psi}_n,\widetilde{\Phi}_n)$ is bounded in $ \mathcal{X}_{p,q}(\mathbb{R}^N)$.
Up to a subsequence, there exists $(\widetilde{\Psi},\widetilde{\Phi})\in \mathcal{X}_{p,q}(\mathbb{R}^N)$ such that $(\widetilde{\Psi}_n,\widetilde{\Phi}_n)\rightharpoonup (\widetilde{\Psi},\widetilde{\Phi})$ in $\mathcal{X}_{p,q}(\mathbb{R}^N)$, $(\widetilde{\Psi}_n,\widetilde{\Phi}_n)\rightarrow (\widetilde{\Psi},\widetilde{\Phi})$ in $L_{loc}^{s}(\mathbb{R}^N)\times L_{loc}^{t}(\mathbb{R}^N)$ for any $(s,t)\in [1,q+1]\times[1,p+1]$, and $(\widetilde{\Psi}_n,\widetilde{\Phi}_n)\rightarrow (\widetilde{\Psi},\widetilde{\Phi})$ almost everywhere in $\mathbb{R}^N$.
   For convenience, we denote $(P_n,K_n)=\mathcal{L}_{\varepsilon_n,\bar{t_n},\bar{\xi_n}}(\Psi_n,\Phi_n)$. Furthermore, by $(P_n,K_n) \in \mathcal{Z}_{\bar{\delta_{n}},\bar{\xi_n}}$, there exist $c_{1n}^0,c_{1n}^1,\cdots,c_{1n}^N$, $c_{2n}^0,c_{2n}^1,\cdots,c_{2n}^N$, $\cdots$, $c_{kn}^0,c_{kn}^1,\cdots,c_{kn}^N$ such that
  \begin{equation}\label{line1}
  (\Psi_n,\Phi_n)-\mathcal{I}^*\big(f'_{\varepsilon_n}(\mathcal{H}_{\bar{\delta_n},\bar{\xi_n}})\Phi_n,g'_{\varepsilon_n}(\mathcal{W}_{\bar{\delta_n},\bar{\xi_n}})\Psi_n\big)=(P_n,K_n)+
  \sum\limits
  _{l=0}^N\sum\limits
  _{m=1}^kc_{mn}^l(\Psi_{\delta_{mn},\xi_{mn}}^l,\Phi_{\delta_{mn},\xi_{mn}}^l),
  \end{equation}
  which also reads
  \begin{align}\label{line1'}
 \left\{
  \begin{array}{ll}
  -\Delta\Psi_n=f'_{\varepsilon_n}(\mathcal{H}_{\bar{\delta_n},\bar{\xi_n}})\Phi_n-\Delta P_n-\sum\limits
  _{l=0}^N\sum\limits
  _{m=1}^kc_{mn}^l\Delta\Psi_{\delta_{mn},\xi_{mn}}^l,
    \quad\mbox{in $\mathbb{R}^N$},\\
   -\Delta\Phi_n=g'_{\varepsilon_n}(\mathcal{W}_{\bar{\delta_n},\bar{\xi_n}})\Psi_n-\Delta K_n-\sum\limits
  _{l=0}^N\sum\limits
  _{m=1}^kc_{mn}^l\Delta\Phi_{\delta_{mn},\xi_{mn}}^l,
   \quad\mbox{in $\mathbb{R}^N$}.
    \end{array}
    \right.
  \end{align}
  Using $(\Psi_n,\Phi_n) \in \mathcal{Z}_{\bar{\delta_{n}},\bar{\xi_n}}$ again, by an easy change of variable, for $i=0,1,\cdots,N$ and $j=1,2\cdots,k$, we have
  \begin{align*}
    0=&\int\limits_{\mathcal{M}} (\nabla _g \Psi_n\cdot \nabla _g \Phi _{\delta_{jn},\xi_{jn}}^i+\nabla _g \Phi_n\cdot \nabla _{g} \Psi _{\delta_{jn},\xi_{jn}}^i+h\Psi_n \Phi _{\delta_{jn},\xi_{jn}}^i+h \Phi_n \Psi _{\delta_{jn},\xi_{jn}}^i)    d v_g\\
    =&\int\limits_{B(0,r_0/\delta_{jn})} \Big[\delta_{jn}^{N-2-\frac{N}{p+1}}\nabla _{g_n} \Psi_n(\exp_{\xi_{jn}}(\delta_{jn} z))\cdot \nabla _{g_n} (\chi_n\Phi _{1,0}^i)+\delta_{jn}^{N-2-\frac{N}{q+1}}\nabla _{g_n} \Phi_n(\exp_{\xi_{jn}}( \delta_{jn} z))\cdot \nabla _{g_n} (\chi_n\Psi _{1,0}^i)\\
    &+\delta_{jn}^{N-\frac{N}{p+1}}h(\exp_{\xi_{jn}}( \delta_{jn} z))\Psi_n (\exp_{\xi_{jn}}( \delta_{jn} z)) \chi_n\Phi _{1,0}^i+\delta_{jn}^{N-\frac{N}{q+1}}h(\exp_{\xi_{jn}}( \delta_{jn} z)) \Phi_n(\exp_{\xi_{jn}}( \delta_{jn} z)) \chi_n\Psi _{1,0}^i\Big]    dz\\
    =&\int\limits_{B(0,r_0/\delta_{jn})} \big[\nabla _{g_n} \widetilde{\Psi}_n\cdot \nabla _{g_n} (\chi_n\Phi _{1,0}^i)+\nabla _{g_n} \widetilde{\Phi}_n\cdot \nabla _{g_n} (\chi_n \Psi _{1,0}^i)+\delta_{jn}^2h_n\widetilde{\Psi}_n \Phi _{1,0}^i+\delta_{jn}^2h_n \widetilde{\Phi}_n \Psi _{1,0}^i\big]    dz,
  \end{align*}
  where $g_{n}(z)=\exp_{\xi_{jn}}^*g(\delta_{jn}z)$, $\chi_{n}(z)=\chi({\delta_{jn}|z|})$ and $h_{n}(z)=h(\exp_{\xi_{jn}}(\delta_{jn}z))$. By Lemma \ref{nonde}, passing to the limit for the above equality, we obtain
  \begin{equation}\label{line2}
  \int\limits_{\mathbb{R}^N}\big(pV_{1,0}^{p-1}\Phi_{1,0}^i\widetilde{\Phi}+qU_{1,0}^{q-1}\Psi_{1,0}^i\widetilde{\Psi}\big)dz=  \int\limits_{\mathbb{R}^N} \big(\nabla  \widetilde{\Psi}\cdot \nabla \Phi _{1,0}^i+\nabla  \widetilde{\Phi}\cdot \nabla  \Psi _{1,0}^i\big)dz=0.
  \end{equation}

  {\bf Step 2:} For any $l=0,1,\cdots,N$ and $m=1,2,\cdots,k$, $c_{mn}^l\rightarrow 0$ as $n\rightarrow \infty$. For any $n\in \mathbb{N}^+$, since $(\Psi_n,\Phi_n)$ and $ (P_n,K_n)$ belong to $ \mathcal{Z}_{\bar{\delta_{n}},\bar{\xi_n}}$, multiplying \eqref{line1} by $(\Psi_{\delta_{jn},\xi_{jn}}^i,\Phi_{\delta_{jn},\xi_{jn}}^i)$, $0\leq i\leq N$, $1\leq j\leq k$, using \eqref{gu1}-\eqref{gu4}, we have
  \begin{align}\label{line3}
    & -\int\limits _{\mathcal{M}} \big(f'_{\varepsilon_n}(\mathcal{H}_{\bar{\delta_n},\bar{\xi_n}})\Phi_n \Phi_{\delta_{jn},\xi_{jn}}^i+g'_{\varepsilon_n}(\mathcal{W}_{\bar{\delta_n},\bar{\xi_n}})\Psi_n \Psi_{\delta_{jn},\xi_{jn}}^i\big)d v_g\nonumber\\
    =&\sum\limits_{l=0}^N\sum\limits_{m=1}^kc_{mn}^l\delta_{il}\delta_{jm}\int\limits_{B(0,r_0/{\delta_{jn}})}\big(p\chi^2_n V_{1,0}^{p-1}(\Phi_{1,0}^i)^2+q\chi^2_n U_{1,0}^{q-1}(\Psi_{1,0}^i)^2
  \big)dz+O(\varepsilon_n).
  \end{align}
  Moreover, by \eqref{line2}, we have
  \begin{align}\label{line4}
    &\int\limits _{\mathcal{M}} \big(f'_{\varepsilon_n}(\mathcal{H}_{\bar{\delta_n},\bar{\xi_n}})\Phi_n \Phi_{\delta_{jn},\xi_{jn}}^i+g'_{\varepsilon_n}(\mathcal{W}_{\bar{\delta_n},\bar{\xi_n}})\Psi_n \Psi_{\delta_{jn},\xi_{jn}}^i\big)d v_g \nonumber\\
    =&\int\limits _{\mathcal{M}} \big((p-\alpha{\varepsilon_n})\mathcal{H}_{\bar{\delta_{n}},\bar{\xi_{n}}}^{p-1-\alpha{\varepsilon_n}}\Phi_n \Phi_{\delta_{jn},\xi_{jn}}^i+(q-\beta{\varepsilon_n})\mathcal{W}_{\bar{\delta_{n}},\bar{\xi_{n}}}^{q-1-\beta{\varepsilon_n}}\Psi_n \Psi_{\delta_{jn},\xi_{jn}}^i\big)d v_g\nonumber\\
    =&\sum\limits_{m=1}^k\int\limits _{\mathcal{M}} \big((p-\alpha{\varepsilon_n})H_{\delta_{mn},\xi_{mn}}^{p-1-\alpha{\varepsilon_n}}\Phi_n \Phi_{\delta_{jn},\xi_{jn}}^i+(q-\beta{\varepsilon_n})W_{\delta_{mn},\xi_{mn}}^{q-1-\beta{\varepsilon_n}}\Psi_n \Psi_{\delta_{jn},\xi_{jn}}^i\big)d v_g\nonumber\\
    =&\delta_{jm}\int\limits _{B(0,r_0/\delta_{jn})} \Big[(p-\alpha{\varepsilon_n})\delta_{jn}^{N-\frac{N(p-\alpha{\varepsilon_n})}{p+1}-\frac{N}{p+1}}(\chi_nV_{1,0})^{p-1-\alpha{\varepsilon_n}}\chi_n\delta_{jn}^
    {\frac{N}{p+1}}\Phi_n(\exp_{\xi_{jn}}(\delta_{jn} z)) \Phi_{1,0}^i\nonumber\\
    &+(q-\beta{\varepsilon_n})\delta_{jn}^{N-\frac{N(q-\beta{\varepsilon_n})}{q+1}-\frac{N}{q+1}}(\chi_nU_{1,0})^{q-1-\beta{\varepsilon_n}}\chi_n\delta_{jn}^{\frac{N}{q+1}}\Psi_n(\exp_{\xi_{jn}}(\delta_{jn} z)) \Psi_{1,0}^i\Big]dz\nonumber\\
    =&\delta_{jm}\int\limits _{B(0,r_0/\delta_{jn})} \Big[(p-\alpha{\varepsilon_n})\delta_{jn}^{N-\frac{N(p-\alpha{\varepsilon_n})}{p+1}-\frac{N}{p+1}}(\chi_nV_{1,0})^{p-1-\alpha{\varepsilon_n}}
    \widetilde{\Phi}_n(z) \Phi_{1,0}^i\nonumber\\
    &+(q-\beta{\varepsilon_n})\delta_{jn}^{N-\frac{N(q-\beta{\varepsilon_n})}{q+1}-\frac{N}{q+1}}(\chi_nU_{1,0})^{q-1-\beta{\varepsilon_n}}
    \widetilde{\Psi}_n(z) \Psi_{1,0}^i\Big]dz
    \nonumber
    \\ \rightarrow &\delta_{jm}\int\limits_{\mathbb{R}^N}(pV_{1,0}^{p-1}\Phi_{1,0}^i\widetilde{\Phi}+qU_{1,0}^{q-1}\Psi_{1,0}^i\widetilde{\Psi})dz=0,\quad \text{as $n\rightarrow+\infty$}.
  \end{align}
It follows from \eqref{line3} and \eqref{line4} that for any $l=0,1,\cdots,N$ and $m=1,2,\cdots,k$, $c_{mn}^l\rightarrow 0$ as $n\rightarrow \infty$.

  {\bf Step 3:} $(\widetilde{\Psi},\widetilde{\Phi})=(0,0)$. For any $j=1,2,\cdots,k$, there hold
  \begin{align*}
    \Delta \widetilde{\Psi}_n=&\delta_{jn}^{\frac{Np}{p+1}}\big[\chi_n\Delta  \Psi_n(\exp _{\xi_{jn}}(\delta_{jn} z))+\nabla \chi_n\cdot \nabla \Psi_n(\exp _{\xi_{jn}}(\delta_{jn} x))+\Psi_n(\exp _{\xi_{jn}}(\delta_{jn} z))\Delta \chi_n\big],
\end{align*}
and
\begin{align*}
    \Delta  \widetilde{\Phi}_n=&\delta_{jn}^{\frac{Nq}{q+1}}\big[\chi_n \Delta  \Phi_n(\exp _{\xi_{jn}}(\delta_{jn} z))+\nabla \chi_n \cdot\nabla \Phi_n(\exp _{\xi_{jn}}(\delta_{jn} x))+\Phi_n(\exp _{\xi_{jn}}(\delta_{jn} z))\Delta \chi_n \big].
\end{align*}
Thus we obtain a system of equations satisfied by $(\widetilde{\Psi}_n,\widetilde{\Phi}_n)$. For any $(\varphi,\psi)\in C_0^\infty(\mathbb{R}^N)\times C_0^\infty(\mathbb{R}^N)$ and $j=1,2,\cdots,k$, by the dominated convergence theorem, we obtain
\begin{equation*}
  \lim\limits_{n\rightarrow+\infty}(p-\alpha\varepsilon)\delta_{jn}^{\frac{Np}{p+1}}\int\limits_{\{z\in \mathbb{R}^N :\varphi(z)\neq0 \}}\big(\chi_n\delta_{jn}^{-\frac{N}{p+1}}V_{1,0}\big)^{p-1-\alpha\varepsilon}\chi_n\Phi_n(\exp_{\xi_{jn}}(\delta_{jn}z))\varphi dz= p\int\limits_{\{x\in \mathbb{R}^N :\varphi(z)\neq0 \}}V_{1,0}^{p-1}\widetilde{\Phi}\varphi dz,
\end{equation*}
and
\begin{equation*}
  \lim\limits_{n\rightarrow+\infty}(q-\beta\varepsilon)\delta_{jn}^{\frac{Nq}{q+1}}\int\limits_{\{z\in \mathbb{R}^N :\psi(z)\neq0 \}}\big(\chi_n\delta_{jn}^{-\frac{N}{q+1}}U_{1,0}\big)^{q-1-\beta\varepsilon}\chi_n\Psi_n(\exp_{\xi_{jn}}(\delta_{jn}z))\psi dz= q\int\limits_{\{z\in \mathbb{R}^N :\psi(z)\neq0 \}}U_{1,0}^{q-1}\widetilde{\Psi}\psi dz.
\end{equation*}
Using \eqref{line1'}, $\|(P_n,K_n)\|\rightarrow0$, $c_{mn}^l\rightarrow 0$ as $n\rightarrow \infty$ for any $l=0,1,\cdots,N$ and $m=1,2,\cdots,k$, we deduce that
$(\widetilde{\Psi},\widetilde{\Phi}) $ satisfies
  \begin{align*}
 \left\{
  \begin{array}{ll}
  -\Delta\widetilde{\Psi}=pV_{1,0}^{p-1}\widetilde{\Phi},
    \quad\mbox{in $\mathbb{R}^N$},\\
   -\Delta\widetilde{\Phi}=qU_{1,0}^{q-1}\widetilde{\Psi},
   \quad\mbox{in $\mathbb{R}^N$}.
    \end{array}
    \right.
  \end{align*}
  This together with \eqref{line2} and Lemma \ref{nonde} yields that $(\widetilde{\Psi},\widetilde{\Phi})=(0,0)$.

  {\bf Step 4:}
  $\|\mathcal{I}^*\big(f'_{\varepsilon_n}(\mathcal{H}_{\bar{\delta_n},\bar{\xi_n}})\Phi_n,g'_{\varepsilon_n}(\mathcal{W}_{\bar{\delta_n},\bar{\xi_n}})\Psi_n\big)\|\rightarrow 0$ as $n\rightarrow \infty$. By \eqref{em}, we know
  \begin{equation*}
    \|\mathcal{I}^*\big(f'_{\varepsilon_n}(\mathcal{H}_{\bar{\delta_n},\bar{\xi_n}})\Phi_n,g'_{\varepsilon_n}(\mathcal{W}_{\bar{\delta_n},\bar{\xi_n}})\Psi_n\big)\|\leq C\big\|f'_{\varepsilon_n}(\mathcal{H}_{\bar{\delta_n},\bar{\xi_n}})\Phi_n\big\|_{\frac{p+1}{p}}+C\big\|g'_{\varepsilon_n}(\mathcal{W}_{\bar{\delta_n},\bar{\xi_n}})\Psi_n\big\|_{\frac{q+1}{q}}.
  \end{equation*}
  For any fixed $R>0$ and $j=1,2,\cdots,k$, by the H\"{o}lder inequality, $\widetilde{\Phi}_n\rightarrow0$ in $L_{loc}^{\frac{p+1}{1+\alpha{\varepsilon_n}}}({\mathbb{R}^N})$ and $\widetilde{\Psi}_n\rightarrow0$ in $L_{loc}^{\frac{q+1}{1+\beta{\varepsilon_n}}}({\mathbb{R}^N})$, we have
  \begin{align*}
   & \big\|f'_{\varepsilon_n}(\mathcal{H}_{\bar{\delta_n},\bar{\xi_n}})\Phi_n\big\|_{\frac{p+1}{p}}^{\frac{p+1}{p}}\\
    =&\int\limits_{\mathcal{M}}\big|(p-\alpha{\varepsilon_n})\mathcal{H}_{\bar{\delta_n},\bar{\xi_n}}^{p-1-\alpha{\varepsilon_n}}\Phi_n\big|^{\frac{p+1}{p}}d v_g\\
    =&\sum\limits_{j=1}^k
    \delta_{jn}^{\frac{N \alpha {\varepsilon_n}}{p}}\int\limits_{B(0,r_0/\delta_{jn})}\big|(p-\alpha{\varepsilon_n})\chi_n^{p-2-\alpha\varepsilon}V_{1,0}^{p-1-\alpha{\varepsilon_n}} \chi_n\delta_{jn}^{\frac{N}{p+1}}\Phi_n(\exp_{\xi_{jn}}(\delta_{jn} z))\big|^{\frac{p+1}{p}}dz\\
    =&\sum\limits_{j=1}^k\delta_{jn}^{\frac{N \alpha {\varepsilon_n}}{p}}\int\limits_{B(0,r_0/\delta_{jn})}\big|(p-\alpha{\varepsilon_n})\chi_n^{p-2-\alpha\varepsilon}V_{1,0}^{p-1-\alpha{\varepsilon_n}} \widetilde{\Phi}_n(z)\big|^{\frac{p+1}{p}}dz\\
    \leq & C\Big(\int\limits_{B(0,r_0/\delta_{jn})}V_{1,0} ^{p+1}dz\Big)^{\frac{p-1-\alpha{\varepsilon_n}}{p}}\Big(\int\limits_{B(0,r_0/\delta_{jn})}|\widetilde{\Phi}_n(z)|^{\frac{p+1}{1+\alpha{\varepsilon_n}}}dz\Big)^{\frac{1+\alpha{\varepsilon_n}}{p}}\\
    \leq &C\Big(\int\limits_{B(0,R)}|\widetilde{\Phi}_n(z)|^{\frac{p+1}{1+\alpha{\varepsilon_n}}}\Big)^{\frac{1+\alpha{\varepsilon_n}}{p}}
    +C\varepsilon_n^{\frac{[(N-2)p-2](p-1-\alpha\varepsilon_n)}{2p}}\rightarrow0,\quad \text{as $n\rightarrow+\infty$},
  \end{align*}
  and
  \begin{align*}
   & \big\|g'_{\varepsilon_n}(\mathcal{W}_{\bar{\delta_n},\bar{\xi_n}})\Psi_n\big\|_{\frac{q+1}{q}}^{\frac{q+1}{q}}\\
    =&\int\limits_{\mathcal{M}}\big|(q-\beta{\varepsilon_n})\mathcal{W}_{\bar{\delta_n},\bar{\xi_n}}^{q-1-\beta{\varepsilon_n}}\Psi_n\big|^{\frac{q+1}{q}}d v_g\\
    =&\sum\limits_{j=1}^k
    \delta_{jn}^{\frac{N \beta {\varepsilon_n}}{q}}\int\limits_{B(0,r_0/\delta_{jn})}\big|(q-\beta{\varepsilon_n})\chi_n^{q-2-\beta\varepsilon}U_{1,0}^{q-1-\beta{\varepsilon_n}} \chi_n\delta_{jn}^{\frac{N}{q+1}}\Psi_n(\exp_{\xi_{jn}}(\delta_{jn} z))\big|^{\frac{q+1}{q}}dz\\
    =&\sum\limits_{j=1}^k\delta_{jn}^{\frac{N \beta {\varepsilon_n}}{q}}\int\limits_{B(0,r_0/\delta_{jn})}\big|(q-\beta{\varepsilon_n})\chi_n^{q-2-\beta\varepsilon}U_{1,0}^{q-1-\beta{\varepsilon_n}} \widetilde{\Psi}_n(z)\big|^{\frac{q+1}{q}}dz\\
    \leq & C\Big(\int\limits_{B(0,r_0/\delta_{jn})}U_{1,0} ^{q+1}dz\Big)^{\frac{q-1-\beta{\varepsilon_n}}{q}}\Big(\int\limits_{B(0,r_0/\delta_{jn})}|\widetilde{\Psi}_n(z)|^{\frac{q+1}{1+\beta{\varepsilon_n}}}dz\Big)^
    {\frac{1+\beta{\varepsilon_n}}{q}}\\
    \leq &
    \left\{
    \begin{array}{ll}
    \displaystyle C\Big(\int\limits_{B(0,R)}|\widetilde{\Psi}_n(z)|^{\frac{q+1}{1+\beta{\varepsilon_n}}}dz\Big)^{\frac{1+\beta{\varepsilon_n}}{q}
    }+C\varepsilon_n^{\frac{[(N-2)q-2](q-1-\beta\varepsilon_n)}{2q}},\quad  \text{if}\ p>\frac{N}{N-2},\\
    \displaystyle C\Big(\int\limits_{B(0,R)}|\widetilde{\Psi}_n(z)|^{\frac{q+1}{1+\beta{\varepsilon_n}}}dz\Big)^{\frac{1+\beta{\varepsilon_n}}{q}
    }+C\varepsilon_n^{\frac{[(N-3)q-3](q-1-\beta\varepsilon_n)}{2q}},\quad   \text{if}\ p=\frac{N}{N-2},\\
    \displaystyle C\Big(\int\limits_{B(0,R)}|\widetilde{\Psi}_n(z)|^{\frac{q+1}{1+\beta{\varepsilon_n}}}dz\Big)^{\frac{1+\beta{\varepsilon_n}}{q}
    }+C\varepsilon_n^{\frac{([(N-2)p-2](q+1)-N)(q-1-\beta\varepsilon_n)}{2q}},\quad \text{if}\ p<\frac{N}{N-2}.\\
    \end{array}
    \right.\\
    \to& 0\ \ \ \text{as}\ n\rightarrow+\infty.
  \end{align*}
From the above arguments,  we
  get $\|(\Psi_n,\Phi_n)\|\rightarrow0$ as $n\rightarrow+\infty$, which is an absurd. Thus, we complete the proof.
\end{proof}
For any $\varepsilon>0$ small enough,  $\bar{t}\in (\mathbb{R}^+)^k$, and $\bar{\xi}\in \mathcal{M}^k$, if $\bar{\delta}$ is as in \eqref{solu'}, then equation \eqref{tou2} is equivalent to
\begin{equation*}
  \mathcal{L}_{\varepsilon,\bar{t},\bar{\xi}}(\Psi,\Phi)=\mathcal{N}_{\varepsilon,\bar{t},\bar{\xi}}(\Psi,\Phi)+\mathcal{R}_{\varepsilon,\bar{t},\bar{\xi}},
\end{equation*}
where
\begin{equation}\label{defn}
\mathcal{N}_{\varepsilon,\bar{t},\bar{\xi}}(\Psi,\Phi)=
  \Pi_{\bar{\delta},\bar{\xi}}^\bot \mathcal{I}^* \big[
  f_\varepsilon(\mathcal{H}_{\bar{\delta},\bar{\xi}}+\Phi)-f_\varepsilon(\mathcal{H}_{\bar{\delta},\bar{\xi}})-f'_\varepsilon(\mathcal{H}_{\bar{\delta},\bar{\xi}})\Phi, g_\varepsilon(\mathcal{W}_{\bar{\delta},\bar{\xi}}+\Psi)-g_\varepsilon(\mathcal{W}_{\bar{\delta},\bar{\xi}})-g'_\varepsilon(\mathcal{W}_{\bar{\delta},\bar{\xi}})\Psi
  \big],
\end{equation}
and
\begin{equation}\label{defr}
\mathcal{R}_{\varepsilon,\bar{t},\bar{\xi}}=
  \Pi_{\bar{\delta},\bar{\xi}}^\bot\Big[
  \mathcal{I}^*\big(f_\varepsilon(\mathcal{H}_{\bar{\delta},\bar{\xi}}),g_\varepsilon(\mathcal{W}_{\bar{\delta},\bar{\xi}})\big)-(\mathcal{W}_{\bar{\delta},\bar{\xi}},
  \mathcal{H}_{\bar{\delta},\bar{\xi}})
  \Big].
\end{equation}
In the following lemma, we estimate the reminder term $\mathcal{R}_{\varepsilon,\bar{t},\bar{\xi}}$.
\begin{lemma}\label{error}
Under the assumptions on $p,q$ and $N$ of Theorem \ref{th}, if  $(\bar{\delta},\bar{\xi})\in \Lambda$ and $\bar{\delta}$ is as in \eqref{solu'}, then for any $\varepsilon>0$ small enough, there holds
\begin{equation*}
  \|\mathcal{R}_{\varepsilon,\bar{t},\bar{\xi}}\|\leq C\varepsilon|\log \varepsilon|,
\end{equation*}
where $\mathcal{R}_{\varepsilon,\bar{t},\bar{\xi}}$ is as in \eqref{defr}.
\end{lemma}
\begin{proof}
By \eqref{em}, we know there exists $C>0$ such that for $\varepsilon>0$ small enough, $\bar{t}\in (\mathbb{R}^+)^k$, and $\bar{\xi}\in \mathcal{M}^k$, there holds
\begin{align*}
  \|\mathcal{R}_{\varepsilon,\bar{t},\bar{\xi}}\|\leq & C\big\|f_\varepsilon(\mathcal{H}_{\bar{\delta},\bar{\xi}})+\Delta_g \mathcal{W}_{\bar{\delta},\bar{\xi}}-h\mathcal{W}_{\bar{\delta},\bar{\xi}}\big\|_{{\frac{p+1}{p}}}
  +C\big\|g_\varepsilon(\mathcal{W}_{\bar{\delta},\bar{\xi}})+\Delta_g \mathcal{H}_{\bar{\delta},\bar{\xi}}-h\mathcal{H}_{\bar{\delta},\bar{\xi}}\big\|_{{\frac{q+1}{q}}}\\
  =&C\sum\limits_{j=1}^k\big\|f_\varepsilon(H_{\delta_j,\xi_j})+\Delta_g W_{\delta_j,\xi_j}-hW_{\delta_j,\xi_j}\big\|_{{\frac{p+1}{p}}}
  +C\sum\limits_{j=1}^k\big\|g_\varepsilon(W_{\delta_j,\xi_j})+\Delta_g H_{\delta_j,\xi_j}-hH_{\delta_j,\xi_j}\big\|_{{\frac{q+1}{q}}}\\
  =&:C\sum\limits_{j=1}^k(I_j+II_j).
\end{align*}
By an easy change of variable, and using Lemma \ref{nonde}, for any $j=1,2,\cdots,k$, we have
\begin{align*}
  I_j^{\frac{p+1}{p}}\leq & C \int\limits_{B(0,r_0/\delta_j)}\big|\delta_j^{\frac{N\alpha \varepsilon}{p+1}}\chi_{\delta_j}^{p-\alpha \varepsilon}V_{1,0}^{p-\alpha \varepsilon}\big|^{\frac{p+1}{p}} d z+C \int\limits_{B(0,r_0/\delta_j)}\big|\chi_{\delta_j}\Delta _{g_{\delta_j,\xi_j}}U_{1,0}\big|^{\frac{p+1}{p}} d z\\
  &+C \int\limits_{B(0,r_0/\delta_j)}\big|\delta^2_jU_{1,0}\Delta _{g_{\delta_j,\xi_j}}\chi_{\delta_j}\big|^{\frac{p+1}{p}} d z+C \int\limits_{B(0,r_0/\delta_j)}\big|\delta_j\nabla_{g_{\delta_j,\xi_j}} \chi_{\delta_j}\cdot \nabla_{g_{\delta_j,\xi_j}} U_{1,0}\big|^{\frac{p+1}{p}} d z\\&+C
  \int\limits_{B(0,r_0/\delta_j)}\big|\delta^2_jh_{\delta_j}\chi_{\delta_j}U_{1,0}\big|^{\frac{p+1}{p}}dz\\
  \leq &C \Big[\int\limits_{B(0,r_0/\delta_j)}\big|\delta_j^{\frac{N\alpha \varepsilon}{p+1}}\chi_{\delta_j}^{p-\alpha \varepsilon}(V_{1,0}^{p-\alpha \varepsilon}-V_{1,0}^{p})\big|^{\frac{p+1}{p}} d z+ \int\limits_{B(0,r_0/\delta_j)}\big|(\delta_j^{\frac{N\alpha \varepsilon}{p+1}}\chi_{\delta_j}^{p-\alpha \varepsilon}-\chi_{\delta_j})V_{1,0}^{p}\big|^{\frac{p+1}{p}} d z\\
  &+ \int\limits_{B(0,r_0/\delta_j)}\big| \chi_{\delta_j}(\Delta _{g_{\delta_j,\xi_j}}U_{1,0}-\Delta _{Eucl}U_{1,0}) \big|^{\frac{p+1}{p}}dx + \int\limits_{B(0,r_0/\delta_j)}\big|\delta^2_jU_{1,0}\Delta _{g_{\delta_j,\xi_j}}\chi_{\delta_j}\big|^{\frac{p+1}{p}} d z\\&+ \int\limits_{B(0,r_0/\delta_j)}\big|\delta_j\nabla_{g_{\delta_j,\xi_j}} \chi_{\delta_j}\cdot \nabla_{g_{\delta_j,\xi_j}} U_{1,0}\big|^{\frac{p+1}{p}} d z+
  \int\limits_{B(0,r_0/\delta_j)}\big|\delta^2_jh_{\delta_j,\xi_j}\chi_{\delta_j}U_{1,0}\big|^{\frac{p+1}{p}}dz\Big]\\
  =:&C(A_1+A_2+A_3+A_4+A_5+A_6),
\end{align*}
where $g_{\delta_j,\xi_j}(z)=\exp_{\xi_j}^*g(\delta_jz)$, $\chi_{\delta_j}(z)=\chi({\delta_j|z|})$ and $h_{\delta_j,\xi_j}(z)=h(\exp_{\xi_j}(\delta_jz))$. We are led to estimate each $A_i$, $i=1,2,\cdots,6$. First, for any fixed $R>0$ large enough and $j=1,2,\cdots,k$, by Lemma \ref{jian1} and Taylor formula, we have
\begin{align*}
  A_1&\leq C \int\limits_{B(0,r_0/\delta_j)}\big|(V_{1,0}^{p-\alpha \varepsilon}-V_{1,0}^{p})\big|^{\frac{p+1}{p}} d z=O\Big(\varepsilon^{\frac{p+1}{p}}\int\limits_{B(0,r_0/\delta_j)}\big |V^{p+1}_{1,0}\log V^{\frac{p+1}{p}}_{1,0}\big|d z\Big)\\
  &=O(\varepsilon^{\frac{p+1}{p}})+ O\Big(\varepsilon^{\frac{p+1}{p}}\int\limits_{B(0,r_0/\delta_j)\backslash B(0,R)}\big |V^{p+1}_{1,0}\log V^{\frac{p+1}{p}}_{1,0}\big|d z\Big)\\
  &=O(\varepsilon^{\frac{p+1}{p}})+ O\Big(\varepsilon^{\frac{p+1}{p}}\int\limits _R^{r_0/\delta_j} r^{N-1-\frac{(N-2)(p+1)^2}{p}}d r\Big)=O(\varepsilon^{\frac{p+1}{p}}),
\end{align*}
as $\varepsilon\rightarrow0$, uniformly with respect to $\xi_j\in \mathcal{M}$ and $t_j\in [a,b]$, $0<a<b<+\infty$,
where we have used the fact that $N<\frac{(N-2)(p+1)^2}{p}$, since $p>\frac{2}{N-2}$.
Using Lemma \ref{jian1} and Taylor formula again, for $j=1,2,\cdots,k$, we obtain
\begin{align*}
  A_2&=O\big( |\varepsilon\log \varepsilon|^{\frac{p+1}{p}}\big)+O\Big( |\varepsilon\log \varepsilon|^{\frac{p+1}{p}} \int\limits_{B(0,r_0/2\delta_j)\backslash B(0,R)}V_{1,0}^{p+1}dz\Big)+O\Big(\int\limits_{B(0,r_0/\delta_j)\backslash B(0,r_0/2\delta_j)}V_{1,0}^{p+1}dz\Big)\\
  &=O\big( |\varepsilon\log \varepsilon|^{\frac{p+1}{p}}\big)+O\Big(|\varepsilon\log \varepsilon|^{\frac{p+1}{p}}\int\limits_{R}^{r_0/2\delta_j}r^{N-1-(N-2)(p+1)}dr\Big)+O\Big(\int\limits_{r_0/2\delta_j}^{r_0/\delta_j}r^{N-1-(N-2)(p+1)}dr\Big)\\
  &=O\big( |\varepsilon\log \varepsilon|^{\frac{p+1}{p}}\big)+O\big(\varepsilon^{\frac{(N-2)p-2}{2}}\big),
\end{align*}
as $\varepsilon\rightarrow0$, uniformly with respect to $\xi_j\in \mathcal{M}$ and $t_j\in [a,b]$.
Since $N\geq8$, then $A_2\leq |\varepsilon\log \varepsilon|^{\frac{p+1}{p}}$.
For any fixed $R>0$ large enough and $j=1,2,\cdots,k$, it follows from \eqref{lap1} and \eqref{lap2} that
\begin{align*}
 A_3=\left\{
  \begin{array}{ll}
  O(\varepsilon^{\frac{p+1}{p}})+\displaystyle O\Big(\varepsilon^{\frac{p+1}{p}}\int\limits _R^{r_0/\delta_j} r^{N-1-\frac{(N-2)(p+1)}{p}}dr\Big)=O(\varepsilon^{\frac{p+1}{p}}),\quad &\text{if $p>\frac{N}{N-2}$;}
 \\ O(\varepsilon^{\frac{p+1}{p}})+\displaystyle O\Big(\varepsilon^{\frac{p+1}{p}}\int\limits _R^{r_0/\delta_j} r^{N-1-\frac{(N-3)(p+1)}{p}}dr\Big)=O(\varepsilon^{\frac{p+1}{p}}),\quad &\text{if $p=\frac{N}{N-2}$;}\\
 O(\varepsilon^{\frac{p+1}{p}})+\displaystyle O\Big(\varepsilon^{\frac{p+1}{p}}\int\limits _R^{r_0/\delta_j} r^{N-1-(N-2)(p+1)+\frac{2p+2}{p}}dr\Big)=O(\varepsilon^{\frac{p+1}{p}}),\quad &\text{if $p<\frac{N}{N-2}$,}
    \end{array}
    \right.
  \end{align*}
as $\varepsilon\rightarrow0$,  uniformly with respect to $\xi_j\in \mathcal{M}$ and $t_j\in [a,b]$,
where we have used the fact that $N\geq8$ and $p>1$. Since there hold $|\chi'_{\delta_j}|\leq C\delta_j$ and $|\chi''_{\delta_j}|\leq C\delta^2_j$ for any $j=1,2,\cdots,k$, we have
\begin{align*}
 A_4=\left\{
  \begin{array}{ll}
  \displaystyle O\Big(\varepsilon^{\frac{2(p+1)}{p}}\int\limits _{r_0/2\delta_j}^{r_0/\delta_j} r^{N-1-\frac{(N-2)(p+1)}{p}}dr\Big)=O(\varepsilon^{\frac{2(p+1)}{p}}),\quad &\text{if $p>\frac{N}{N-2}$;}
 \\ \displaystyle O\Big(\varepsilon^{\frac{2(p+1)}{p}}\int\limits _{r_0/2\delta_j}^{r_0/\delta_j} r^{N-1-\frac{(N-3)(p+1)}{p}}dr\Big)=O(\varepsilon^{\frac{2(p+1)}{p}}),\quad &\text{if $p=\frac{N}{N-2}$;}\\
 \displaystyle O\Big(\varepsilon^{\frac{2(p+1)}{p}}\int\limits _{r_0/2\delta_j}^{r_0/\delta_j} r^{N-1-(N-2)(p+1)+\frac{2p+2}{p}}dr\Big)=O(\varepsilon^{\frac{2(p+1)}{p}}),\quad &\text{if $p<\frac{N}{N-2}$,}
    \end{array}
    \right.
  \end{align*}
  and
 \begin{align*}
 A_5=\left\{
  \begin{array}{ll}
  \displaystyle O\Big(\varepsilon^{\frac{p+1}{p}}\int\limits _{r_0/2\delta_j}^{r_0/\delta_j} r^{N-1-\frac{(N-1)(p+1)}{p}}dr\Big)=O(\varepsilon^{\frac{p+1}{p}}),\quad &\text{if $p>\frac{N}{N-2}$;}
 \\ \displaystyle O\Big(\varepsilon^{\frac{p+1}{p}}\int\limits _{r_0/2\delta_j}^{r_0/\delta_j} r^{N-1-\frac{(N-2)(p+1)}{p}}dr\Big)=O(\varepsilon^{\frac{p+1}{p}}),\quad &\text{if $p=\frac{N}{N-2}$;}\\
 \displaystyle O\Big(\varepsilon^{\frac{p+1}{p}}\int\limits _{r_0/2\delta_j}^{r_0/\delta_j} r^{N-1-(N-2)(p+1)+\frac{p+1}{p}}dr\Big)=O(\varepsilon^{\frac{p+1}{p}}),\quad &\text{if $p<\frac{N}{N-2}$,}
    \end{array}
    \right.
  \end{align*}
as $\varepsilon\rightarrow0$, uniformly with respect to $\xi_j\in \mathcal{M}$ and $t_j\in [a,b]$. Moreover, for any fixed $R>0$ large enough and $j=1,2,\cdots,k$, it's easy to obtain
 \begin{align*}
 A_6=\left\{
  \begin{array}{ll}
  O(\varepsilon^{\frac{p+1}{p}})+\displaystyle O\Big(\varepsilon^{\frac{p+1}{p}}\int\limits _R^{r_0/\delta_j} r^{N-1-\frac{(N-2)(p+1)}{p}}dr\Big)=O(\varepsilon^{\frac{p+1}{p}}),\quad &\text{if $p>\frac{N}{N-2}$;}
 \\ O(\varepsilon^{\frac{p+1}{p}})+\displaystyle O\Big(\varepsilon^{\frac{p+1}{p}}\int\limits _R^{r_0/\delta_j} r^{N-1-\frac{(N-3)(p+1)}{p}}dr\Big)=O(\varepsilon^{\frac{p+1}{p}}),\quad &\text{if $p=\frac{N}{N-2}$;}\\
 O(\varepsilon^{\frac{p+1}{p}})+\displaystyle O\Big(\varepsilon^{\frac{p+1}{p}}\int\limits _R^{r_0/\delta_j} r^{N-1-(N-2)(p+1)+\frac{2p+2}{p}}dr\Big)=O(\varepsilon^{\frac{p+1}{p}}),\quad &\text{if $p<\frac{N}{N-2}$,}
    \end{array}
    \right.
  \end{align*}
as $\varepsilon\rightarrow0$,   uniformly with respect to $\xi_j\in \mathcal{M}$ and $t_j\in [a,b]$.
From the above arguments, we obtain $I_j=O(\varepsilon |\log \varepsilon|)$ for any $j=1,2,\cdots,k$.

Similarly, we can prove that
\begin{align*}
  II_j^{\frac{q+1}{q}}\leq & C\Big[ \int\limits_{B(0,r_0/\delta_j)}\big|\delta_j^{\frac{N\beta \varepsilon}{q+1}}\chi_{\delta_j}^{q-\beta \varepsilon}(U_{1,0}^{q-\beta \varepsilon}-U_{1,0}^{q})\big|^{\frac{q+1}{q}} d z+ \int\limits_{B(0,r_0/\delta_j)}\big|(\delta_j^{\frac{N\beta \varepsilon}{q+1}}\chi_{\delta_j}^{q-\beta \varepsilon}-\chi_{\delta_j})U_{1,0}^{q}\big|^{\frac{q+1}{q}} d z\\
  &+ \int\limits_{B(0,r_0/\delta_j)}\big| \chi_{\delta_j}(\Delta _{g_{\delta_j,\xi_j}}V_{1,0}-\Delta _{Eucl}V_{1,0}) \big|^{\frac{q+1}{q}}d z+ \int\limits_{B(0,r_0/\delta_j)}\big|\delta^2_jV_{1,0}\Delta _{g_{\delta_j,\xi_j}}\chi_{\delta_j}\big|^{\frac{q+1}{q}} d z\\&+ \int\limits_{B(0,r_0/\delta_j)}\big|\delta_j\nabla_{g_{\delta_j,\xi_j}} \chi_{\delta_j}\cdot \nabla_{g_{\delta_j,\xi_j}} V_{1,0}\big|^{\frac{q+1}{q}} d z+
  \int\limits_{B(0,r_0/\delta_j)}\big|\delta^2_jh_{\delta_j,\xi_j}\chi_{\delta_j}V_{1,0}\big|^{\frac{q+1}{q}}dz\Big]\\
  =:&C(B_1+B_2+B_3+B_4+B_5+B_6).
\end{align*}
For any fixed $R>0$ large enough and $j=1,2,\cdots,k$,  by $N\geq8$ and $q>1$,  we have
 \begin{align*}
 B_1=\left\{
  \begin{array}{ll}
  O(\varepsilon^{\frac{q+1}{q}})+\displaystyle O\Big(\varepsilon^{\frac{q+1}{q}}\int\limits _R^{r_0/\delta_j} r^{N-1-\frac{(N-2)(q+1)^2}{q}}dr\Big)=O(\varepsilon^{\frac{q+1}{q}}),\quad &\text{if $p>\frac{N}{N-2}$;}
 \\ O(\varepsilon^{\frac{q+1}{q}})+\displaystyle O\Big(\varepsilon^{\frac{q+1}{q}}\int\limits _R^{r_0/\delta_j} r^{N-1-\frac{(N-3)(q+1)^2}{q}}dr\Big)=O(\varepsilon^{\frac{q+1}{q}}),\quad &\text{if $p=\frac{N}{N-2}$;}\\
 O(\varepsilon^{\frac{q+1}{q}})+\displaystyle O\Big(\varepsilon^{\frac{q+1}{q}}\int\limits _R^{r_0/\delta_j} r^{N-1-\frac{[(N-2)p-2](q+1)^2}{q}}dr\Big)=O(\varepsilon^{\frac{q+1}{q}}),\quad &\text{if $p<\frac{N}{N-2}$,}
    \end{array}
    \right.
  \end{align*}
  and
 \begin{align*}
 B_2=\left\{
  \begin{array}{ll}
  O\big( |\varepsilon\log \varepsilon|^{\frac{q+1}{q}}\big)+\displaystyle O\Big(\int\limits_{r_0/2\delta_j}^{r_0/\delta_j}r^{N-1-(N-2)(q+1)}dr\Big)=O\big( |\varepsilon\log \varepsilon|^{\frac{q+1}{q}}\big),\,\, &\text{if $p>\frac{N}{N-2}$;}
 \\ O\big( |\varepsilon\log \varepsilon|^{\frac{q+1}{q}}\big)+\displaystyle O\Big(\int\limits_{r_0/2\delta_j}^{r_0/\delta_j}r^{N-1-(N-3)(q+1)}dr\Big)=O\big( |\varepsilon\log \varepsilon|^{\frac{q+1}{q}}\big),\,\, &\text{if $p=\frac{N}{N-2}$, $N\geq 10$;}\\
 O\big( |\varepsilon\log \varepsilon|^{\frac{q+1}{q}}\big)+\displaystyle O\Big(\int\limits_{r_0/2\delta_j}^{r_0/\delta_j}r^{N-1-[(N-2)p-2](q+1)}dr\Big)=O\big( |\varepsilon\log \varepsilon|^{\frac{q+1}{q}}\big),\,\, &\text{if $p<\frac{N}{N-2}$, $N\geq 12$,}
    \end{array}
    \right.
  \end{align*}
as $\varepsilon\rightarrow0$, uniformly with respect to $\xi_j\in \mathcal{M}$ and $t_j\in [a,b]$. Similar arguments as above,  we have
\begin{equation*}
  B_3=O(\varepsilon^{\frac{q+1}{q}})+\displaystyle O\Big(\varepsilon^{\frac{q+1}{q}}\int\limits _R^{r_0/\delta_j} r^{N-1-\frac{(N-2)(q+1)}{q}}dr\Big)=O(\varepsilon^{\frac{q+1}{q}}),
\end{equation*}
\begin{equation*}
  B_4=\displaystyle O\Big(\varepsilon^{\frac{2(q+1)}{q}}\int\limits _{r_0/2\delta_j}^{r_0/\delta_j} r^{N-1-\frac{(N-2)(q+1)}{q}}dr\Big)=O(\varepsilon^{\frac{2(q+1)}{q}}),
\end{equation*}
\begin{equation*}
  B_5=\displaystyle O\Big(\varepsilon^{\frac{q+1}{q}}\int\limits _{r_0/2\delta_j}^{r_0/\delta_j} r^{N-1-\frac{(N-1)(q+1)}{q}}dr\Big)=O(\varepsilon^{\frac{q+1}{q}}),
\end{equation*}
and
\begin{equation*}
  B_6=O(\varepsilon^{\frac{q+1}{q}})+\displaystyle O\Big(\varepsilon^{\frac{q+1}{q}}\int\limits _R^{r_0/\delta_j} r^{N-1-\frac{(N-2)(q+1)}{q}}dr\Big)=O(\varepsilon^{\frac{q+1}{q}}),
\end{equation*}
as $\varepsilon\rightarrow0$, uniformly with respect to $\xi_j\in \mathcal{M}$ and $t_j\in [a,b]$. Hence $II_j=O(\varepsilon |\log \varepsilon|)$ for any $j=1,2,\cdots,k$. This ends the proof.
\end{proof}

We now prove Proposition \ref{propo1} by using Lemmas  \ref{line} and \ref{error}.\\
{\bf Proof of Proposition \ref{propo1}.} For any $\varepsilon>0$ small enough, $\bar{t}\in (\mathbb{R}^+)^k$, and $\bar{\xi}\in \mathcal{M}^k$, if $\bar{\delta}$ is as in \eqref{solu'}, we define the map $\mathcal{T}_{\varepsilon,\bar{t},\bar{\xi}}:\mathcal{Z}_{\bar{\delta},\bar{\xi}}\rightarrow \mathcal{Z}_{\bar{\delta},\bar{\xi}}$ by
\begin{equation*}
  \mathcal{T}_{\varepsilon,\bar{t},\bar{\xi}}(\Psi,\Phi)=\mathcal{L}^{-1}_{\varepsilon,\bar{t},\bar{\xi}}(\mathcal{N}_{\varepsilon,\bar{t},\bar{\xi}}(\Psi,\Phi)+\mathcal{R}_{\varepsilon,\bar{t},\bar{\xi}}),
\end{equation*}
where $\mathcal{L}_{\varepsilon,\bar{t},\bar{\xi}}$, $\mathcal{N}_{\varepsilon,\bar{t},\bar{\xi}}$ and $\mathcal{R}_{\varepsilon,\bar{t},\bar{\xi}}$ are as in \eqref{defl}, \eqref{defn} and \eqref{defr}, respectively. We also set
\begin{equation*}
  \mathcal{B}_{\varepsilon,\bar{t},\bar{\xi}}(\gamma)=\big\{(\Psi,\Phi)\in \mathcal{Z}_{\bar{\delta},\bar{\xi}}:\|(\Psi,\Phi)\|\leq \gamma \|\mathcal{R}_{\varepsilon,\bar{t},\bar{\xi}}\|\big\},
\end{equation*}
where $\gamma>0$ is a fixed constant large enough. We prove that the map $\mathcal{T}_{\varepsilon,\bar{t},\bar{\xi}}$ admits a fixed point  $(\Psi_{\varepsilon,\bar{t},\bar{\xi}},\Phi_{\varepsilon,\bar{t},\bar{\xi}})$. Therefore, we shall prove that,
for any $\varepsilon>0$ small,
there hold:

$(i)$ $\mathcal{T}_{\varepsilon,\bar{t},\bar{\xi}}(\mathcal{B}_{\varepsilon,\bar{t},\bar{\xi}}(\gamma))\subset \mathcal{B}_{\varepsilon,\bar{t},\bar{\xi}}(\gamma)$;

$(ii)$ $\mathcal{T}_{\varepsilon,\bar{t},\bar{\xi}}$ is a contraction map on $\mathcal{B}_{\varepsilon,\bar{t},\bar{\xi}}(\gamma)$.

For $(i)$, by \eqref{em} and Lemma \ref{line}, for any $\varepsilon>0$ small enough,
and $(\Psi,\Phi)\in \mathcal{B}_{\varepsilon,\bar{t},\bar{\xi}}(\gamma)$, we have
\begin{align*}
  &\|\mathcal{T}_{\varepsilon,\bar{t},\bar{\xi}}(\mathcal{B}_{\varepsilon,\bar{t},\bar{\xi}}(\gamma))\|
  \leq C\|\mathcal{N}_{\varepsilon,\bar{t},\bar{\xi}}(\Psi,\Phi)\|+C\|\mathcal{R}_{\varepsilon,\bar{t},\bar{\xi}}\|\\
   \leq& C\Big[\|\mathcal{R}_{\varepsilon,\bar{t},\bar{\xi}}\|+ \big\|f_\varepsilon(\mathcal{H}_{\bar{\delta},\bar{\xi}}+\Phi)-f_\varepsilon(\mathcal{H}_{\bar{\delta},\bar{\xi}})-f'_\varepsilon(\mathcal{H}_{\bar{\delta},\bar{\xi}})\Phi\big\|_{\frac{p+1}{p}}\\
  &+\big\|g_\varepsilon(\mathcal{W}_{\bar{\delta},\bar{\xi}}+\Psi)-g_\varepsilon(\mathcal{W}_{\bar{\delta},\bar{\xi}})-g'_\varepsilon(\mathcal{W}_{\bar{\delta},\bar{\xi}})\Psi\big\|_{\frac{q+1}{q}}\Big]=:C(\|\mathcal{R}_{\varepsilon,\bar{t},\bar{\xi}}\|+I+II).
\end{align*}
By the mean value formula, Lemmas \ref{gs}, \ref{error}, and the Sobolev embedding theorem, we obtain
\begin{equation*}
  I\leq C\|\Phi\|^{p-\alpha\varepsilon}_{\frac{(p+1)(p-\alpha\varepsilon)}{p}}\leq C\|\Phi\|^{p-\alpha\varepsilon}_{p+1}\leq C\gamma^{p-\alpha\varepsilon}\|\mathcal{R}_{\varepsilon,\bar{t},\bar{\xi}}\|^{p-\alpha\varepsilon}
  \leq \gamma\|\mathcal{R}_{\varepsilon,\bar{t},\bar{\xi}}\|,
\end{equation*}
and
\begin{align*}
  II&\leq
  \left\{
    \begin{array}{ll}
  C\|\Psi\|^{q-\beta\varepsilon}_{\frac{(q+1)(q-\beta\varepsilon)}{q}}+C\|\Psi\|_{\frac{2(q+1)}{2+\beta\varepsilon}}^2\sum\limits_{j=1}^k\|W_{\delta_j,\xi_j}\|_{q+1}^{q-2-\beta\varepsilon}\leq \gamma\|\mathcal{R}_{\varepsilon,\bar{t},\bar{\xi}}\|,\quad &\text{if $q> 2$},\nonumber\\
  C\|\Psi\|^{q-\beta\varepsilon}_{\frac{(q+1)(q-\beta\varepsilon)}{q}}\leq \gamma\|\mathcal{R}_{\varepsilon,\bar{t},\bar{\xi}}\|,\quad &\text{if $q\leq 2$},
  \end{array}
    \right.
\end{align*}
where we have used the fact that $\|W_{\delta_j,\xi_j}\|_{q+1}<+\infty$ for any $1<p\leq \frac{N+2}{N-2}\leq q$ and $j=1,2,\cdots,k$. So we have $(i)$.

Similarly, by \eqref{em} and Lemma \ref{line}, for any $\varepsilon>0$ small enough,
and $(\Psi_1,\Phi_1), (\Psi_2,\Phi_2)\in \mathcal{B}_{\varepsilon,\bar{t},\bar{\xi}}(\gamma)$, we have
\begin{align*}
  &\|\mathcal{T}_{\varepsilon,\bar{t},\bar{\xi}}(\Psi_1,\Phi_1)-\mathcal{T}_{\varepsilon,\bar{t},\bar{\xi}}(\Psi_2,\Phi_2)\| \\
  \leq& C\|\mathcal{N}_{\varepsilon,\bar{t},\bar{\xi}}(\Psi_1,\Phi_1)-\mathcal{N}_{\varepsilon,\bar{t},\bar{\xi}}(\Psi_2,\Phi_2)\|\\
   \leq& C\Big[\big\|f_\varepsilon(\mathcal{H}_{\bar{\delta},\bar{\xi}}+\Phi_1)-f_\varepsilon(\mathcal{H}_{\bar{\delta},\bar{\xi}}+\Phi_2)
   -f'_\varepsilon(\mathcal{H}_{\bar{\delta},\bar{\xi}})(\Phi_1-\Phi_2)\big\|_{\frac{p+1}{p}}\\
  &+C\big\|g_\varepsilon(\mathcal{W}_{\bar{\delta},\bar{\xi}}+\Psi_1)-g_\varepsilon(\mathcal{W}_{\bar{\delta},\bar{\xi}}+\Psi_2)-g'_\varepsilon(\mathcal{W}_{\bar{\delta},\bar{\xi}})(\Psi_1-\Psi_2)\big\|_{\frac{q+1}{q}}\Big]=:C(III+IV).
\end{align*}
By the mean value formula, Lemma \ref{gs}, and the Sobolev embedding theorem, we obtain
\begin{align}\label{one}
  III&\leq  C\Big(\|\Phi_1\|_{\frac{(p-1-\alpha\varepsilon)(p+1)}{p-1}}^{p-1-\alpha\varepsilon}
  +\|\Phi_2\|_{\frac{(p-1-\alpha\varepsilon)(p+1)}{p-1}}^{p-1-\alpha\varepsilon}\Big)\|\Phi_1-\Phi_2\|_{p+1}\nonumber\\
  &\leq C \gamma^{p-1-\alpha\varepsilon}\|\mathcal{R}_{\varepsilon,\bar{t},\bar{\xi}}\|^{p-1-\alpha\varepsilon}\|\Phi_1-\Phi_2\|,
\end{align}
and
\begin{align}\label{two}
IV&\leq \left\{
    \begin{array}{ll}
    C\Big(\|\Psi_1\|_{\frac{(q-1-\beta\varepsilon)(q+1)}{q-1}}^{q-1-\beta\varepsilon}
  +\|\Psi_2\|_{\frac{(q-1-\beta\varepsilon)(q+1)}{q-1}}^{q-1-\beta\varepsilon}\Big)\|\Psi_1-\Psi_2\|_{q+1}\\
  \quad+
  C(\|\Psi_1\|_{\frac{q+1}{1+\beta\varepsilon}}+\|\Psi_2\|_{\frac{q+1}{1+\beta\varepsilon}})\|\Psi_1-\Psi_2\|_{q+1}\sum\limits_{j=1}^k\|W_{\delta_j,\xi_j}\|_{q+1}^{q-2-\beta\varepsilon}
    \quad &\text{if $q>2$},\nonumber\\
   C\Big(\|\Psi_1\|_{\frac{(q-1-\beta\varepsilon)(q+1)}{q-1}}^{q-1-\beta\varepsilon}
  +\|\Psi_2\|_{\frac{(q-1-\beta\varepsilon)(q+1)}{q-1}}^{q-1-\beta\varepsilon}\Big)\|\Psi_1-\Psi_2\|_{q+1},\quad &\text{if $q\leq2$},
  \end{array}
  \right.\\
  &\leq \left\{
    \begin{array}{ll}
    C \gamma^{q-1-\beta\varepsilon}\|\mathcal{R}_{\varepsilon,\bar{t},\bar{\xi}}\|^{q-1-\beta\varepsilon}\|\Psi_1-\Psi_2\|+C \gamma\|\mathcal{R}_{\varepsilon,\bar{t},\bar{\xi}}\|\|\Psi_1-\Psi_2\|,\quad &\text{if $q>2$},\\
    C \gamma^{q-1-\beta\varepsilon}\|\mathcal{R}_{\varepsilon,\bar{t},\bar{\xi}}\|^{q-1-\beta\varepsilon}\|\Psi_1-\Psi_2\|,\quad &\text{if $q\leq2$}.
    \end{array}
  \right.
\end{align}
By Lemma \ref{error}, we know $C \gamma\|\mathcal{R}_{\varepsilon,\bar{t},\bar{\xi}}\|,C \gamma^{p-1-\alpha\varepsilon}\|\mathcal{R}_{\varepsilon,\bar{t},\bar{\xi}}\|^{p-1-\alpha\varepsilon},C \gamma^{q-1-\beta\varepsilon}\|\mathcal{R}_{\varepsilon,\bar{t},\bar{\xi}}\|^{q-1-\beta\varepsilon}\in (0,1)$. This proves $(ii)$. Finally, by using the implicit function theorem, we can prove the regularity of $(\Psi_{\varepsilon,\bar{t},\bar{\xi}},\Phi_{\varepsilon,\bar{t},\bar{\xi}})$ with respect to $\bar{t}$ and $\bar{\xi}$. Thus we complete the proof.
\qed

\section{Proof of Proposition \ref{propo2}}\label{sec5}

This section is devoted to the proof of Proposition \ref{propo2}. As a first step, we have

\begin{lemma}
Under the assumptions on $p,q$ and $N$ of Theorem \ref{th}, if $\bar{\delta}$ is as in \eqref{solu'}, then for any $\varepsilon>0$ small enough, if $(\bar{t},\bar{\xi})$ is a critical point of the functional $\widetilde{\mathcal{J}}_\varepsilon$, then $\big(\mathcal{W}_{\bar{\delta},\bar{\xi}}+\Psi_{\varepsilon,\bar{t},\bar{\xi}}
  ,\mathcal{H}_{\bar{\delta},\bar{\xi}}+\Phi_{\varepsilon,\bar{t},\bar{\xi}}\big)$ is a  solution of system \eqref{pro}, or equivalently of  \eqref{repro}.
\end{lemma}
\begin{proof}
Let $(\bar{t},\bar{\xi})$ be a critical point of  $\widetilde{\mathcal{J}}_\varepsilon$, where $\bar{t}=(t_1,t_2,\cdots,t_k)\in (\mathbb{R}^+)^k$ and $\bar{\xi}=(\xi_1,\xi_2,\cdots,\xi_k)\in \mathcal{M}^k$. Let $\bar{\xi}(y)=\big(\exp_{\xi_1}(y^1),\exp_{\xi_2}(y^2),\cdots,\exp_{\xi_k}(y^k)\big)$, $y=(y^1,y^2,\cdots,y^k)\in B(0,r)^k$, and $\xi_j(y^j)=\exp_{\xi_j}(y^j)$ for any $j=1,2,\cdots,k$, then $\bar{\xi}(0)=\bar{\xi}$. Since $(\bar{t},\bar{\xi})$ be a critical point of  $\widetilde{\mathcal{J}}_\varepsilon$, for any $m=1,2,\cdots,k$ and $l=1,2,\cdots,N$, there hold
\begin{equation*}
  \mathcal{J}'_{\varepsilon}\big(\mathcal{W}_{\bar{\delta},\bar{\xi}}+\Psi_{\varepsilon,\bar{t},\bar{\xi}},\mathcal{H}_{\bar{\delta},\bar{\xi}}+\Phi_{\varepsilon,\bar{t},\bar{\xi}}\big)
  \big(\partial _{t_m}\mathcal{W}_{\bar{\delta},\bar{\xi}}+\partial _{t_m}\Psi_{\varepsilon,\bar{t},\bar{\xi}},\partial _{t_m}\mathcal{H}_{\bar{\delta},\bar{\xi}}+\partial _{t_m}\Phi_{\varepsilon,\bar{t},\bar{\xi}}\big)=0,
\end{equation*}
and
\begin{equation*}
  \mathcal{J}'_{\varepsilon}\big(\mathcal{W}_{\bar{\delta},\bar{\xi}}+\Psi_{\varepsilon,\bar{t},\bar{\xi}},\mathcal{H}_{\bar{\delta},\bar{\xi}}+\Phi_{\varepsilon,\bar{t},\bar{\xi}}\big)
  \big(\partial _{y^m_l}\mathcal{W}_{\bar{\delta},\bar{\xi}}+\partial _{y^m_l}\Psi_{\varepsilon,\bar{t},\bar{\xi}},\partial _{y^m_l}\mathcal{H}_{\bar{\delta},\bar{\xi}}+\partial _{y^m_l}\Phi_{\varepsilon,\bar{t},\bar{\xi}}\big)=0.
\end{equation*}
For any $(\varphi,\psi)\in \mathcal{X}_{p,q}(\mathcal{M})$, by Proposition \ref{propo1}, there exist some constants $c_{10},c_{11},\cdots,c_{1N}$,  $c_{20},c_{21},\cdots,c_{2N}$, $\cdots$, $c_{k0},c_{k1},\cdots,c_{kN}$ such that
\begin{equation*}
  \mathcal{J}_\varepsilon'(\mathcal{W}_{\bar{\delta},\bar{\xi}}+\Psi_{\varepsilon,\bar{t},\bar{\xi}},\mathcal{H}_{\delta,\xi}+\Phi_{\varepsilon,\bar{t},\bar{\xi}})(\varphi,\psi)=
  \sum\limits_{l=0}^N\sum\limits_{m=1}^kc_{lm}\big\langle(\Psi_{\delta_m,\xi_m}^l,\Phi^l_{\delta_m,\xi_m}),(\varphi,\psi)\big\rangle_h.
\end{equation*}
Let $\partial _s$ denote $\partial_ {t_m}$ or $\partial _{y^m_l}$ for any $m=1,2,\cdots,k$ and $l=1,2,\cdots,N$. Then
\begin{align}\label{com}
  \partial_s \widetilde{\mathcal{J}}_{\varepsilon}(\bar{t},\bar{\xi}(y)) =&\mathcal{J}'_{\varepsilon}\big(\mathcal{W}_{\bar{\delta},\bar{\xi}(y)}+\Psi_{\varepsilon,\bar{t},\bar{\xi}(y)},\mathcal{H}_{\bar{\delta},\bar{\xi}(y)}+\Phi_{\varepsilon,\bar{t},\bar{\xi}(y)}\big)
  \big(\partial _{s}\mathcal{W}_{\bar{\delta},\bar{\xi}(y)}+\partial _{s}\Psi_{\varepsilon,\bar{t},\bar{\xi}(y)},\partial _{s}\mathcal{H}_{\bar{\delta},\bar{\xi}(y)}+\partial _{s}\Phi_{\varepsilon,\bar{t},\bar{\xi}(y)}\big)\nonumber\\
  =& \big\langle \big(\mathcal{W}_{\bar{\delta},\bar{\xi}(y)}+\Psi_{\varepsilon,\bar{t},\bar{\xi}(y)}
  ,\mathcal{H}_{\bar{\delta},\bar{\xi}(y)}+\Phi_{\varepsilon,\bar{t},\bar{\xi}(y)}\big)
  -\mathcal{I}^*\big(f_\varepsilon(\mathcal{H}_{\bar{\delta},\bar{\xi}(y)}+\Phi_{\varepsilon,\bar{t},\bar{\xi}(y)}),g_\varepsilon(\mathcal{W}_{\bar{\delta},\bar{\xi}(y)}+\Psi_{\varepsilon,\bar{t},\bar{\xi}(y)})\big),\nonumber\\
  &\big(\partial _{s}\mathcal{W}_{\bar{\delta},\bar{\xi}(y)}+\partial _{s}\Psi_{\varepsilon,\bar{t},\bar{\xi}(y)},\partial _{s}\mathcal{H}_{\bar{\delta},\bar{\xi}(y)}+\partial _{s}\Phi_{\varepsilon,\bar{t},\bar{\xi}(y)}\big)\big\rangle \nonumber\\
  =&\sum\limits_{i=0}^N\sum\limits_{j=1}^kc_{ij}\big\langle\big(\Psi_{\delta_j,\xi_j(y^j)}^i,\Phi^i_{\delta_j,\xi_j(y^j)}\big),\big(\partial _{s}\mathcal{W}_{\bar{\delta},\bar{\xi}(y)}+\partial _{s}\Psi_{\varepsilon,\bar{t},\bar{\xi}(y)},\partial _{s}\mathcal{H}_{\bar{\delta},\bar{\xi}(y)}+\partial _{s}\Phi_{\varepsilon,\bar{t},\bar{\xi}(y)}\big)\big\rangle_h.
\end{align}
We prove that if we compute \eqref{com} at $y=0$, then for any $\varepsilon>0$ small enough, there holds
\begin{equation*}
  c_{ij}=0,\quad \text{for any  $i=0,1,\cdots,N$ and $j=1,2,\cdots,k$}.
\end{equation*}
Since $(\bar{t},\bar{\xi})$ is a critical point of  $\widetilde{\mathcal{J}}_\varepsilon$, then
\begin{equation}\label{e50}
  \partial_s \widetilde{\mathcal{J}}_{\varepsilon}(\bar{t},\bar{\xi}(y))|_{y=0}=0.
\end{equation}
For any $m=1,2,\cdots,k$ and $l=1,2,\cdots,N$,
  we can easily check that there hold
  \begin{equation}\label{ch1}
    \big(\partial_{t_m}\mathcal{W}_{\bar{\delta},\bar{\xi}},\partial_{t_m}\mathcal{H}_{\bar{\delta},\bar{\xi}}\big)=-\frac{1}{2t_m}\big(\Psi^0_{\delta_m,\xi_m},\Phi^0_{\delta_m,\xi_m}\big),
  \end{equation}
  and
  \begin{equation}\label{ch2}
    \big(\partial_{y^m_l} (\mathcal{W}_{\bar{\delta},\bar{\xi}(y))})\big|_{y=0},
    \partial_{y^m_l} (\mathcal{H}_{\bar{\delta},\bar{\xi}(y)})\big|_{y=0}\big)=\frac{1}{\delta_m}\big(\Psi^l_{\delta_m,\xi_m}+R_1,\Phi^l_{\delta_m,\xi_m}+R_2\big),
  \end{equation}
  where $\|(R_1,R_2)\|=o(\varepsilon^{\frac{\vartheta}{2}})$ as $\varepsilon\rightarrow0$ for all $\vartheta\in (0,1)$. Using \eqref{gu1}-\eqref{gu4}, we have
\begin{align}\label{e51}
  &\sum\limits_{i=0}^N\sum\limits_{j=1}^kc_{ij}\big\langle\big(\Psi_{\delta_j,\xi_j}^i,\Phi^i_{\delta_j,\xi_j}\big),\big(\partial _{t_m}\mathcal{W}_{\bar{\delta},\bar{\xi}},\partial _{t_m}\mathcal{H}_{\bar{\delta},\bar{\xi}}\big)\big\rangle_h\nonumber\\
  =&-\frac{1}{2t_m}\sum\limits_{i=0}^N\sum\limits_{j=1}^kc_{ij}\big\langle\big(\Psi_{\delta_j,\xi_j}^i,\Phi^i_{\delta_j,\xi_j}\big),\big(\Psi^0_{\delta_m,\xi_m},\Phi^0_{\delta_m,\xi_m}\big)\big\rangle_h\nonumber\\
  =&-\frac{1}{2t_m}\sum\limits_{i=0}^N\sum\limits_{j=1}^kc_{ij}\delta_{i0}\delta_{jm}\int\limits_{B(0,r_0/{\delta_m})}\big(p\chi^2_{\delta_m} V_{1,0}^{p-1}(\Phi_{1,0}^0)^2+q\chi^2_{\delta_m} U_{1,0}^{q-1}(\Psi_{1,0}^0)^2
  \big)dx+O(\varepsilon),
\end{align}
\begin{align}\label{e52}
  &\sum\limits_{i=0}^N\sum\limits_{j=1}^kc_{ij}\big\langle\big(\Psi_{\delta_j,\xi_j}^i,\Phi^i_{\delta_j,\xi_j}\big),\big(\partial _{y^m_l}\mathcal{W}_{\bar{\delta},\bar{\xi}(y)}\big|_{y=0},\partial _{y^m_l}\mathcal{H}_{\bar{\delta},\bar{\xi}(y)}\big|_{y=0}\big)\big\rangle_h\nonumber\\
  =&\frac{1}{\delta_m}\sum\limits_{i=0}^N\sum\limits_{j=1}^kc_{ij}\big\langle\big(\Psi_{\delta_j,\xi_j}^i,\Phi^i_{\delta_j,\xi_j}\big),\big(\Psi^l_{\delta_m,\xi_m}+R_1,\Phi^l_{\delta_m,\xi_m}+R_2\big)\big\rangle_h\nonumber\\
  =&\frac{1}{\delta_m}\sum\limits_{i=0}^N\sum\limits_{j=1}^kc_{ij}\delta_{il}\delta_{jm}\int\limits_{B(0,r_0/{\delta_m})}\big(p\chi^2_{\delta_m} V_{1,0}^{p-1}(\Phi_{1,0}^l)^2+q\chi^2_{\delta_m} U_{1,0}^{q-1}(\Psi_{1,0}^l)^2
  \big)dx+O(\varepsilon),
\end{align}
and
\begin{align}\label{e53}
  &\sum\limits_{i=0}^N\sum\limits_{j=1}^kc_{ij}\big\langle\big(\Psi_{\delta_j,\xi_j}^i,\Phi^i_{\delta_j,\xi_j}\big),\big(\partial _{s}\Psi_{\varepsilon,\bar{t},\bar{\xi}(y)}\big|_{y=0},\partial _{s}\Phi_{\varepsilon,\bar{t},\bar{\xi}(y)}\big|_{y=0}\big)\big\rangle_h\nonumber\\
  =&-\sum\limits_{i=0}^N\sum\limits_{j=1}^kc_{ij}\big\langle\big(\partial_s\Psi_{\delta_j,\xi_j(y^j)}^i\big|_{y=0},\partial_s\Phi^i_{\delta_j,\xi_j(y^j)}\big|_{y=0}\big),
  \big(\Psi_{\varepsilon,\bar{t},\bar{\xi}},\Phi_{\varepsilon,\bar{t},\bar{\xi}}\big)\big\rangle_h,
\end{align}
where $\chi_{\delta_m}(x)=\chi(\delta_m|x|)$. For any $\vartheta\in (0,1)$, with the aid of Proposition \ref{propo1}, it's easy to check
\begin{align}\label{e54}
 &\sum\limits_{i=0}^N\sum\limits_{j=1}^kc_{ij}\big\langle\big(\partial _{t_m}\Psi_{\delta_j,\xi_j}^i,\partial _{t_m}\Phi^i_{\delta_j,\xi_j}\big),\big(\Psi_{\varepsilon,\bar{t},\bar{\xi}},\Phi_{\varepsilon,\bar{t},\bar{\xi}}\big)\big\rangle_h\nonumber\\
  \leq&\frac{1}{2t_m}\sum\limits_{i=0}^N\sum\limits_{j=1}^kc_{ij}\delta_{jm}\Big(\big\|\partial_\delta\big(\delta^{-\frac{N}{q+1}}\Psi^i_{1,0}(\delta^{-1}y)\big)\big|_{\delta=1}\big\|_{\dot{W}^{1,p^*}(\mathbb{R}^N)}\|\nabla _g \Phi_{\varepsilon,\bar{t},\bar{\xi}}\|_{q^*}\nonumber\\
 &+\big\|\partial_\delta\big(\delta^{-\frac{N}{p+1}}\Phi^i_{1,0}(\delta^{-1}y)\big)\big|_{\delta=1}\big\|_{\dot{W}^{1,q^*}(\mathbb{R}^N)}\|\nabla _g \Psi_{\varepsilon,\bar{t},\bar{\xi}}\|_{p^*}\Big)\nonumber\\=&o\big(\varepsilon^{\vartheta}\big),
\end{align}
and
\begin{align}\label{e55}
&\sum\limits_{i=0}^N\sum\limits_{j=1}^kc_{ij}\big\langle\big(\partial _{y^m_l}\Psi_{\delta_j,\xi_j(y^j)}^i\big|_{y=0},\partial _{y^m_l}\Phi^i_{\delta_j,\xi_j(y^j)}\big|_{y=0}\big),
   \big(\Psi_{\varepsilon,\bar{t},\bar{\xi}},\Phi_{\varepsilon,\bar{t},\bar{\xi}}\big)\big\rangle_h\nonumber\\
  \leq&\frac{1}{\delta_m}\sum\limits_{i=0}^N\sum\limits_{j=1}^kc_{ij}\delta_{jm}\Big(\big\|\partial_{y_l} \Psi^i_{1,0} \big\|_{\dot{W}^{1,p^*}(\mathbb{R}^N)}\|\nabla _g \Phi_{\varepsilon,\bar{t},\bar{\xi}}\|_{q^*}+\big\|\partial _{y_l}\Phi^i_{1,0}\big\|_{\dot{W}^{1,q^*}(\mathbb{R}^N)}\|\nabla _g \Psi_{\varepsilon,\bar{t},\bar{\xi}}\|_{p^*}\Big)\nonumber\\
 =&o\big(\varepsilon^{\vartheta}\big).
\end{align}
Therefore, by \eqref{e50} and \eqref{e51}-\eqref{e55}, we deduce that the linear system in \eqref{com} has only a trivial solution when $y=0$ provided that $\varepsilon>0$ small enough. This ends the proof.
\end{proof}

In the next lemma, we give the asymptotic expansion of $ \mathcal{J}_{\varepsilon}(\mathcal{W}_{\bar{\delta},\bar{\xi}},\mathcal{H}_{\bar{\delta},\bar{\xi}})$ as $\varepsilon\rightarrow0$ for $(\bar{\delta},\bar{\xi})\in \Lambda$, where $\bar{\delta}$ is as in \eqref{solu'}.

\begin{lemma}\label{e5.1}
Under the assumptions on $p,q$ and $N$ of Theorem \ref{th}, if  $(\bar{\delta},\bar{\xi})\in \Lambda$ and $\bar{\delta}$ is as in \eqref{solu'}, then 
there holds
\begin{align*}
  \mathcal{J}_{\varepsilon}(\mathcal{W}_{\bar{\delta},\bar{\xi}},\mathcal{H}_{\bar{\delta},\bar{\xi}})=&\frac{2k}{N}L_1+c_1\varepsilon-c_2\varepsilon  \log \varepsilon +\Psi_k(\bar{t},\bar{\xi})\varepsilon +o(\varepsilon),
\end{align*}
as $\varepsilon\rightarrow0$, $C^1$-uniformly with respect to $\bar{\xi}$ in $\mathcal{M}^k$ and to $\bar{t}$ in compact subsets of $(\mathbb{R}^+)^k$, where the function $\Psi_k(\bar{t},\bar{\xi})$ is defined as \eqref{deffa}, $c_1$ and $c_2$ are given in (\ref{defc1c2}).
\end{lemma}

\begin{proof}
For any $\xi\in \mathcal{M}$, there holds
\begin{equation*}
  \frac{1}{\omega_{N-1}r^{N-1}}\int\limits_{\partial B(\xi,r)}d\sigma_g=1-\frac{1}{6N}Scal_g(\xi)r^2+O(r^4)
\end{equation*}
as $r\rightarrow0$, where $\omega_{N-1}$ is  the volume of the unit sphere in $\mathbb{R}^N$. Furthermore, by standard properties of the exponential map, the reminder $O(r^4)$ can be made $C^1$-uniform with respect to $\xi$. Under the assumptions on $p,q$ and $N$ of Theorem \ref{th},  we can compute
\begin{align}\label{ener1}
  &\int\limits_{\mathcal{M}}\nabla _g \Big(\sum \limits_{j=1}^kW_{\delta_j,\xi_j}\Big)\cdot \nabla _g \Big(\sum \limits_{j=1}^kH_{\delta_j,\xi_j}\Big)d v_g=\sum \limits_{j=1}^k\int\limits_{\mathcal{M}}\nabla _g W_{\delta_j,\xi_j}\cdot \nabla _g H_{\delta_j,\xi_j}d v_g\nonumber\\
  =&\sum \limits_{j=1}^k\Big[\int\limits_{B(0,r_0/2
  \delta_j)}\nabla _{g_{\delta_j,\xi_j}}U_{1,0}\cdot \nabla _{g_{\delta_j,\xi_j}}V_{1,0}\Big(1-\frac{1}{6N}Scal_g(\xi_j)\delta_j^2|z|^2+O(\delta_j^4|z|^4)\Big)dz\nonumber\\
  &+\int\limits_{B(r_0/
  \delta_j)\backslash B(r_0/2
  \delta_j)}\nabla _{g_{\delta_j,\xi_j}}(\chi_{\delta_j}U_{1,0})\cdot \nabla _{g_{\delta_j,\xi_j}}(\chi_{\delta_j}V_{1,0})\Big(1-\frac{1}{6N}Scal_g(\xi_j)\delta^2_j|z|^2+O(\delta^4_j|z|^4)\Big)dz\Big]\nonumber\\
  =&\int\limits_{\mathbb{R}^N}\nabla _{g_{\delta_j,\xi_j}}U_{1,0}\cdot \nabla _{g_{\delta_j,\xi_j}}V_{1,0}\Big(1-\frac{1}{6N}Scal_g(\xi_j)\delta^2_j|z|^2+O(\delta^4_j|z|^4)\Big)dz\nonumber\\
  &-\int\limits_{B^c(0,r_0/2
  \delta_j)}\nabla _{g_{\delta_j,\xi_j}}U_{1,0}\cdot \nabla _{g_{\delta_j,\xi_j}}V_{1,0}\Big(1-\frac{1}{6N}Scal_g(\xi_j)\delta^2_j|z|^2+O(\delta^4_j|z|^4)\Big)dz\nonumber\\
  &+\int\limits_{B(r_0/
  \delta_j)\backslash B(r_0/2
  \delta_j)}\nabla _{g_{\delta_j,\xi_j}}(\chi_{\delta_j}U_{1,0})\cdot \nabla _{g_{\delta_j,\xi_j}}(\chi_{\delta_j}V_{1,0})\Big(1-\frac{1}{6N}Scal_g(\xi_j)\delta^2_j|z|^2+O(\delta^4_j|z|^4)\Big)dz\nonumber\\
  =&kL_1-\sum \limits_{j=1}^k\Big\{\frac{L_2Scal_g(\xi_j)}{6N}\delta^2_j+o(\delta^2_j)\Big\},
\end{align}
\begin{equation}\label{ener2}
 \frac{d}{dt} \Big\{\int\limits_{\mathcal{M}}\nabla _g \Big(\sum \limits_{j=1}^kW_{\delta_j,\xi_j}\Big)\cdot \nabla _g \Big(\sum \limits_{j=1}^kH_{\delta_j,\xi_j}\Big)d v_g\Big\}=-\sum \limits_{j=1}^k\Big\{\frac{L_2Scal_g(\xi_j)}{3N}\delta_j\delta'_j+o(\delta_j\delta'_j)\Big\},
\end{equation}
and
\begin{align}\label{ener3}
  &\int\limits_{\mathcal{M}}h\Big(\sum \limits_{j=1}^kW_{\delta_j,\xi_j}\Big)\Big(\sum \limits_{j=1}^kH_{\delta_j,\xi_j}\Big)d v_g=
  \sum \limits_{j=1}^k\int\limits_{\mathcal{M}}hW_{\delta_j,\xi_j}H_{\delta_j,\xi_j}d v_g\nonumber\\=&
  \sum \limits_{j=1}^k\Big\{\delta^2_j\int\limits_{\mathbb{R}^N}h_{\delta_j,\xi_j}U_{1,0} V_{1,0}\big(1+\delta^2_j|z|^2\big)dz-\delta^2_j\int\limits_{B^c(0,r_0/2
  \delta_j)}h_{\delta_j,\xi_j}U_{1,0} V_{1,0}\big(1+\delta^2_j|z|^2\big)dz\nonumber\\
  &+\delta_j^2\int\limits_{B(r_0/
  \delta_j)\backslash B(r_0/2
  \delta_j)}h_{\delta_j,\xi_j}\chi^2_{\delta_j}U_{1,0} V_{1,0}\big(1+\delta^2_j|z|^2\big)dz\Big\}\nonumber\\
  =&\sum \limits_{j=1}^k\big\{L_3h(\xi_j)\delta^2_j+o(\delta^2_j)\big\},
\end{align}
\begin{equation}\label{ener4}
  \frac{d}{dt}\Big\{\int\limits_{\mathcal{M}}h\Big(\sum \limits_{j=1}^kW_{\delta_j,\xi_j}\Big)\Big(\sum \limits_{j=1}^kH_{\delta_j,\xi_j}\Big)d v_g\Big\}=\sum \limits_{j=1}^k\big\{2L_3h(\xi_j)\delta_j\delta'_j+o(\delta_j\delta'_j)\big\},
\end{equation}
as $\varepsilon\rightarrow0$, $C^1$-uniformly with respect to $\bar{\xi}$ in $\mathcal{M}^k$ and to $\bar{t}$ in compact subsets of $(\mathbb{R}^+)^k$, where $g_{\delta_j,\xi_j}(z)=\exp_{\xi_j}^*g(\delta_jz)$, $\chi_{\delta_j}(z)=\chi({\delta_j|z|})$, and $h_{\delta_j,\xi_j}(z)=h(\exp_{\xi_j}(\delta_jz))$.
Using the Taylor formula, we have
\begin{align}\label{ener5}
 & \frac{1}{p+1-\alpha\varepsilon}\int\limits_{\mathcal{M}}
 \Big(\sum \limits_{j=1}^kH_{\delta_j,\xi_j}\Big)^{p+1-\alpha\varepsilon}d v_g=\sum \limits_{j=1}^k\frac{1}{p+1-\alpha\varepsilon}\int\limits_{\mathcal{M}}
 H_{\delta_j,\xi_j}^{p+1-\alpha\varepsilon}d v_g\nonumber\\=&\sum \limits_{j=1}^k\Big\{\frac{1}{p+1}\int\limits_{\mathcal{M}}
  H_{\delta_j,\xi_j}^{p+1}d v_g+\alpha\varepsilon \int\limits_{\mathcal{M}}
  \Big[\frac{H_{\delta_j,\xi_j}^{p+1}}{(p+1)^2}-\frac{H_{\delta_j,\xi_j}^{p+1}\log H_{\delta_j,\xi_j}}{p+1}\Big]d v_g+o(\delta^2_j)\Big\}\nonumber\\
  =&\sum \limits_{j=1}^k\Big\{\frac{1}{p+1}\int\limits_{\mathbb{R}^N}
  V_{1,0}^{p+1}\Big(1-\frac{1}{6N}Scal_g(\xi_j)\delta^2_j|z|^2+O(\delta^4_j|z|^4)\Big)dz\nonumber\\&-\frac{1}{p+1}\int\limits_{B^c(0,r_0/2
  \delta_j)}
  V_{1,0}^{p+1}\Big(1-\frac{1}{6N}Scal_g(\xi_j)\delta^2_j|z|^2+O(\delta^4_j|z|^4)\Big)dz\nonumber\\
  &+\frac{1}{p+1}\int\limits_{B(r_0/
  \delta_j)\backslash B(r_0/2
  \delta_j)}\chi_{\delta_j}^{p+1}V_{1,0}^{p+1}\Big(1-\frac{1}{6N}Scal_g(\xi_j)\delta^2_j|z|^2+O(\delta^4_j|z|^4)\Big)dz\nonumber\\
  &+\frac{\alpha\varepsilon}{(p+1)^2}\int\limits_{\mathbb{R}^N}V_{1,0}^{p+1}\big(1+\delta^2_j|z|^2\big)dz-\frac{\alpha\varepsilon}{(p+1)^2}\int\limits_{B^c(0,r_0/2
  \delta_j)}V_{1,0}^{p+1}\big(1+\delta^2_j|z|^2\big)dz\nonumber\\
  &+\frac{\alpha\varepsilon}{(p+1)^2}\int\limits_{B(r_0/
  \delta_j)\backslash B(r_0/2
  \delta_j)}\chi_{\delta_j}^{p+1}V_{1,0}^{p+1}\big(1+\delta^2_j|z|^2\big)dz\nonumber\\
  &-\frac{\alpha\varepsilon}{p+1}\int\limits_{\mathbb{R}^N}V_{1,0}^{p+1}\log \big(\delta_j^{-\frac{N}{p+1}}V_{1,0}\big)\big(1+\delta^2_j|z|^2\big)dz+\frac{\alpha\varepsilon}{p+1}\int\limits_{B^c(0,r_0/2
  \delta_j)}V_{1,0}^{p+1}\log \big(\delta_j^{-\frac{N}{p+1}}V_{1,0}\big)\big(1+\delta^2_j|z|^2\big)dz\nonumber\\
  &-\frac{\alpha\varepsilon}{p+1}\int\limits_{B(r_0/
  \delta_j)\backslash B(r_0/2
  \delta_j)}\chi_{\delta_j}^{p+1}V_{1,0}^{p+1}\log \big(\chi_{\delta_j}\delta_j^{-\frac{N}{p+1}}V_{1,0}\big)\big(1+\delta^2_j|z|^2\big)dz+o(\delta^2_j)\Big\}\nonumber\\
  =&\frac{kL_1}{p+1}+\frac{kL_1\alpha }{(p+1)^2} \varepsilon-\frac{kL_6\alpha}{p+1}\varepsilon+\sum \limits_{j=1}^k\Big\{-\frac{L_4Scal_g(\xi_j)}{6N(p+1)}\delta^2_j+\frac{NL_1\alpha}{(p+1)^2}\varepsilon\log {\delta_j}+o(\delta^2_j)\Big\},
\end{align}
and
\begin{equation}\label{ener6}
  \frac{d}{dt}\Big\{\frac{1}{p+1-\alpha\varepsilon}\int\limits_{\mathcal{M}}
  \Big(\sum \limits_{j=1}^kH_{\delta_j,\xi_j}\Big)^{p+1-\alpha\varepsilon}d v_g\Big\}=\sum \limits_{j=1}^k\Big\{-\frac{L_4Scal_g(\xi_j)}{3N(p+1)}\delta_j\delta'_j+\frac{NL_1\alpha\delta'_j\varepsilon}{(p+1)^2\delta_j} +o(\delta_j\delta'_j)\Big\},
\end{equation}
as $\varepsilon\rightarrow0$, $C^1$-uniformly with respect to $\bar{\xi}$ in $\mathcal{M}^k$ and to $\bar{t}$ in compact subsets of $(\mathbb{R}^+)^k$.
Similarly, we can prove that
\begin{align}\label{ener7}
  &\frac{1}{q+1-\beta\varepsilon}\int\limits_{\mathcal{M}}
  \Big(\sum \limits_{j=1}^kW_{\delta_j,\xi_j}\Big)^{q+1-\beta\varepsilon}d v_g\nonumber\\
  =&\frac{kL_1}{q+1}+\frac{kL_1\beta }{(q+1)^2} \varepsilon-\frac{kL_7\beta }{q+1}\varepsilon+\sum \limits_{j=1}^k\Big\{-\frac{L_5Scal_g(\xi_j)}{6N(q+1)}\delta^2_j+\frac{NL_1\beta}{(q+1)^2}\varepsilon\log {\delta_j}+o(\delta^2_j)\Big\},
\end{align}
and
\begin{equation}\label{ener8}
  \frac{d}{dt}\Big(\frac{1}{q+1-\beta\varepsilon}\int\limits_{\mathcal{M}}
  \Big(\sum \limits_{j=1}^kW_{\delta_j,\xi_j}\Big)^{q+1-\beta\varepsilon}d v_g\Big)=\sum \limits_{j=1}^k\Big\{-\frac{L_5Scal_g(\xi_j)}{3N(q+1)}\delta_j\delta'_j+\frac{NL_1\beta\delta'_j\varepsilon}{(q+1)^2\delta_j} +o(\delta_j\delta'_j)\Big\},
\end{equation}
as $\varepsilon\rightarrow0$, $C^1$-uniformly with respect to $\bar{\xi}$ in $\mathcal{M}^k$ and to $\bar{t}$ in compact subsets of $(\mathbb{R}^+)^k$, where we have used the fact that $N\geq10$ if $p=\frac{N}{N-2}$ and $N\geq12$ if $p<\frac{N}{N-2}$.
From \eqref{ener1}-\eqref{ener8}, we conclude the result.
\end{proof}
We now give the asymptotic expansion of the function $\widetilde{\mathcal{J}}_\varepsilon$ defined in \eqref{defj} as $\varepsilon\rightarrow0$.
\begin{lemma}\label{topr}
Under the assumptions on $p,q$ and $N$ of Theorem \ref{th}, if  $(\bar{\delta},\bar{\xi})\in \Lambda$ and $\bar{\delta}$ is as in \eqref{solu'}, then there holds
\begin{align*}
 \widetilde{\mathcal{J}}_{\varepsilon}(\bar{t},\bar{\xi})= \mathcal{J}_{\varepsilon}(\mathcal{W}_{\bar{\delta},\bar{\xi}},\mathcal{H}_{\bar{\delta},\bar{\xi}})+o(\varepsilon),
\end{align*}
as $\varepsilon\rightarrow0$,  $C^0$-uniformly with respect to $\bar{\xi}$ in $\mathcal{M}^k$ and to $\bar{t}$ in compact subsets of $(\mathbb{R}^+)^k$.
\end{lemma}
\begin{proof}
It's easy to verify
\begin{align*}
  &\widetilde{\mathcal{J}}_{\varepsilon}(\bar{t},\bar{\xi})- \mathcal{J}_{\varepsilon}(\mathcal{W}_{\bar{\delta},\bar{\xi}},\mathcal{H}_{\bar{\delta},\bar{\xi}})\\
  =&\int\limits_{\mathcal{M}}\big(-\Delta_g \mathcal{W}_{\bar{\delta},\bar{\xi}}+h\mathcal{W}_{\bar{\delta},\bar{\xi}}-f_\varepsilon(\mathcal{H}_{\bar{\delta},\bar{\xi}})\big)\Phi_{\varepsilon,\bar{t},\bar{\xi}}d v_g+\int\limits_{\mathcal{M}}\big(-\Delta_g \mathcal{H}_{\bar{\delta},\bar{\xi}}+h\mathcal{H}_{\bar{\delta},\bar{\xi}}-g_\varepsilon(\mathcal{W}_{\bar{\delta},\bar{\xi}})\big)\Psi_{\varepsilon,\bar{t},\bar{\xi}}d v_g\\
  &+\int\limits_{\mathcal{M}}\big(\nabla_g \Psi_{\varepsilon,\bar{t},\bar{\xi}}\cdot \nabla _g \Phi_{\varepsilon,\bar{t},\xi}+h\Psi_{\varepsilon,\bar{t},\bar{\xi}}\Phi_{\varepsilon,\bar{t},\bar{\xi}}\big)dv_g-\int\limits_{\mathcal{M}}\big(F_\varepsilon(\mathcal{H}_{\bar{\delta},\bar{\xi}}+\Phi_{\varepsilon,\bar{t},\xi})-
  F_\varepsilon(\mathcal{H}_{\bar{\delta},\bar{\xi}})-f_\varepsilon(\mathcal{H}_{\bar{\delta},\bar{\xi}}) \Phi_{\varepsilon,\bar{t},\bar{\xi}}\big)dv_g\\
  &-\int\limits_{\mathcal{M}}\big(G_\varepsilon(\mathcal{W}_{\bar{\delta},\bar{\xi}}+\Psi_{\varepsilon,\bar{t},\xi})-
  G_\varepsilon(\mathcal{W}_{\bar{\delta},\bar{\xi}})-g_\varepsilon(\mathcal{W}_{\bar{\delta},\bar{\xi}}) \Psi_{\varepsilon,\bar{t},\bar{\xi}}\big)dv_g,
\end{align*}
where $F_\varepsilon(u)=\int\limits_{0}^uf_\varepsilon(s)ds$, $G_\varepsilon(u)=\int\limits_{0}^ug_\varepsilon(s)ds$. By the H\"{o}lder inequality, Proposition \ref{propo1}, Lemma \ref{error}, and \eqref{embedd}, for any $\vartheta\in (0,1)$, we get
\begin{align*}
  &\int\limits_{\mathcal{M}}\big(-\Delta_g  \mathcal{W}_{\bar{\delta},\bar{\xi}}+h\mathcal{W}_{\bar{\delta},\bar{\xi}}-f_\varepsilon(\mathcal{H}_{\bar{\delta},\bar{\xi}})\big)\Phi_{\varepsilon,\bar{t},\bar{\xi}}d v_g \leq
  \big\|-\Delta_g \mathcal{W}_{\bar{\delta},\bar{\xi}}+h\mathcal{W}_{\bar{\delta},\bar{\xi}}-f_\varepsilon(\mathcal{H}_{\bar{\delta},\bar{\xi}})\big\|_{\frac{p+1}{p}}\|\Phi_{\varepsilon,\bar{t},\bar{\xi}}\|_{p+1}=o(\varepsilon^{2\vartheta}),
\end{align*}
\begin{align*}
  &\int\limits_{\mathcal{M}}\big(-\Delta_g  \mathcal{H}_{\bar{\delta},\bar{\xi}}+h\mathcal{H}_{\bar{\delta},\bar{\xi}}-g_\varepsilon(\mathcal{W}_{\bar{\delta},\bar{\xi}})\big)\Psi_{\varepsilon,\bar{t},\xi}d v_g
  \leq
  \big\|-\Delta_g \mathcal{H}_{\bar{\delta},\bar{\xi}}+h\mathcal{H}_{\bar{\delta},\bar{\xi}}-g_\varepsilon(\mathcal{W}_{\bar{\delta},\bar{\xi}})\big\|_{\frac{q+1}{q}}\|\Psi_{\varepsilon,\bar{t},\bar{\xi}}\|_{q+1}=o(\varepsilon^{2\vartheta}),
\end{align*}
and
\begin{align*}
  &\int\limits_{\mathcal{M}}\big(\nabla_g \Psi_{\varepsilon,\bar{t},\bar{\xi}}\cdot \nabla _g \Phi_{\varepsilon,\bar{t},\bar{\xi}}+h\Psi_{\varepsilon,\bar{t},\bar{\xi}}\Phi_{\varepsilon,\bar{t},\bar{\xi}}\big)dv_g
   \leq \|\nabla_g \Psi_{\varepsilon,\bar{t},\bar{\xi}}\|_{p^*}
  \|\nabla_g \Phi_{\varepsilon,\bar{t},\bar{\xi}}\|_{q^*}+C\|\Psi_{\varepsilon,\bar{t},\bar{\xi}}\|_2\|\Phi_{\varepsilon,\bar{t},\bar{\xi}}\|_2=o(\varepsilon^{2\vartheta}),
\end{align*}
as $\varepsilon\rightarrow0$,  uniformly with respect to $\bar{\xi}$ in $\mathcal{M}^k$ and to $\bar{t}$ in compact subsets of $(\mathbb{R}^+)^k$.
Moreover, by the mean value formula, Lemma \ref{gs}, \eqref{ener5}, \eqref{ener7} and the Sobolev embedding theorem, for any $\vartheta\in (0,1)$, we obtain
\begin{align}\label{three}
  &\int\limits_{\mathcal{M}}\big(F_\varepsilon(\mathcal{H}_{\bar{\delta},\bar{\xi}}+\Phi_{\varepsilon,\bar{t},\bar{\xi}})-
  F_\varepsilon(\mathcal{H}_{\bar{\delta},\bar{\xi}})-f_\varepsilon(\mathcal{H}_{\bar{\delta},\bar{\xi}}) \Phi_{\varepsilon,\bar{t},\bar{\xi}}\big)dv_g
  \leq C\int\limits_{\mathcal{M}}\mathcal{H}_{\bar{\delta},\bar{\xi}}^{p-1-\alpha\varepsilon}\Phi^2_{\varepsilon,\bar{t},\bar{\xi}}dv_g+
  C\int\limits_{\mathcal{M}}\Phi_{\varepsilon,\bar{t},\bar{\xi}}^{p+1-\alpha\varepsilon}dv_g\nonumber\\
  \leq& C\|\Phi_{\varepsilon,\bar{t},\bar{\xi}}\|_{p+1-\alpha\varepsilon}^2\sum\limits_{j=1}^k\|H_{\delta_j,\xi_j}\|_{p+1-\alpha\varepsilon}^{p-1-\alpha\varepsilon}+C\|\Phi_{\varepsilon,\bar{t},\bar{\xi}}\|_{p+1-\alpha\varepsilon}^{p+1-\alpha\varepsilon}=o(\varepsilon^{2\vartheta}),
\end{align}
and
\begin{align}\label{four}
  &\int\limits_{\mathcal{M}}\big(G_\varepsilon(\mathcal{W}_{\bar{\delta},\bar{\xi}}+\Psi_{\varepsilon,\bar{t},\bar{\xi}})-
  G_\varepsilon(\mathcal{W}_{\bar{\delta},\bar{\xi}})-g_\varepsilon(\mathcal{W}_{\bar{\delta},\bar{\xi}}) \Psi_{\varepsilon,\bar{t},\bar{\xi}}\big)dv_g
  \leq C\int\limits_{\mathcal{M}}\mathcal{W}_{\bar{\delta},\bar{\xi}}^{q-1-\beta\varepsilon}\Psi^2_{\varepsilon,\bar{t},\bar{\xi}}dv_g+
  C\int\limits_{\mathcal{M}}\Psi_{\varepsilon,\bar{t},\bar{\xi}}^{q+1-\beta\varepsilon}dv_g\nonumber\\
  \leq& C\|\Psi_{\varepsilon,\bar{t},\bar{\xi}}\|_{q+1-\beta\varepsilon}^2\sum\limits_{j=1}^k\|W_{\delta_j,\xi_j}\|_{q+1-\beta\varepsilon}^{q-1-\beta\varepsilon}+
  C\|\Psi_{\varepsilon,\bar{t},\bar{\xi}}\|_{q+1-\beta\varepsilon}^{q+1-\beta\varepsilon}=o(\varepsilon^{2\vartheta}),
\end{align}
as $\varepsilon\rightarrow0$,  uniformly with respect to $\bar{\xi}$ in $\mathcal{M}^k$ and to $\bar{t}$ in compact subsets of $(\mathbb{R}^+)^k$. This ends the proof.
\end{proof}

Next, we estimate the gradient of the reduced energy.
\begin{lemma}\label{topr'}
Under the assumptions on $p,q$ and $N$ of Theorem \ref{th}, if  $(\bar{\delta},\bar{\xi})\in \Lambda$ and $\bar{\delta}$ is as in \eqref{solu'}, then for any $m=1,2,\cdots,k$, there holds
\begin{align*}
 \partial_{t_m}\widetilde{\mathcal{J}}_{\varepsilon}(\bar{t},\bar{\xi})=\partial_{t_m} \Psi_k(\bar{t},\bar{\xi})+o(\varepsilon),
\end{align*}
and set $\bar{\xi}(y)=\big(\exp_{\xi_1}(y^1),\exp_{\xi_2}(y^2),\cdots,\exp_{\xi_k}(y^k)\big)$, $y=(y^1,y^2,\cdots,y^k)\in B(0,r)^k$, for any $l=1,2,\cdots,N$, it holds that
\begin{equation*}
  \partial_{y^m_l}\widetilde{\mathcal{J}}_{\varepsilon}(\bar{t},\bar{\xi}(y))\big|_{y=0}=\partial_{y^m_l}\Psi_k(\bar{t},\bar{\xi}(y))\big|_{y=0}+o(\varepsilon),
\end{equation*}
as $\varepsilon\rightarrow0$,  $C^0$-uniformly with respect to $\bar{\xi}$ in $\mathcal{M}^k$ and to $\bar{t}$ in compact subsets of $(\mathbb{R}^+)^k$, where the function $\Psi_k(\bar{t},\bar{\xi})$ is defined as \eqref{deffa}.
\end{lemma}
\begin{proof}
For any $(\varphi,\psi)\in \mathcal{X}_{p,q}(\mathcal{M})$, by Proposition \ref{propo1}, there exist $c_{10},c_{11},\cdots,c_{1N}$,  $c_{20},c_{21},\cdots,c_{2N}$, $\cdots$, $c_{k0},c_{k1},\cdots,c_{kN}$ such that
\begin{equation}\label{real}
  \mathcal{J}_\varepsilon'(\mathcal{W}_{\bar{\delta},\bar{\xi}}+\Psi_{\varepsilon,\bar{t},\bar{\xi}},\mathcal{H}_{\delta,\xi}+\Phi_{\varepsilon,\bar{t},\bar{\xi}})(\varphi,\psi)=
  \sum\limits_{l=0}^N\sum\limits_{m=1}^kc_{lm}\big\langle\big(\Psi_{\delta_m,\xi_m}^l,\Phi^l_{\delta_m,\xi_m}\big),(\varphi,\psi)\big\rangle_h.
\end{equation}
We claim that: for any $\vartheta\in (0,1)$, there holds
\begin{equation}\label{cla}
  \sum\limits_{l=0}^N\sum\limits_{m=1}^k|c_{lm}|=O(\varepsilon^{\vartheta}).
\end{equation}
Taking $(\varphi,\psi)=(\Psi_{\delta_j,\xi_j}^i,\Phi^i_{\delta_j,\xi_j})$, $0\leq i\leq N$, $1\leq j\leq k$, by \eqref{gu1}-\eqref{gu4}, we have
  \begin{align}\label{take1}
  &\sum\limits_{l=0}^N\sum\limits_{m=1}^kc_{lm}\big\langle\big(\Psi_{\delta_m,\xi_m}^l,\Phi^l_{\delta_m,\xi_m}\big),\big(\Psi_{\delta_j,\xi_j}^i,\Phi^i_{\delta_j,\xi_j}\big)\big\rangle_h
    \nonumber\\=&\sum\limits_{l=0}^N\sum\limits_{m=1}^kc_{lm}\delta_{il}\delta_{jm}\int\limits_{B(0,r_0/{\delta_j})}\big(p\chi^2_{\delta_j} V_{1,0}^{p-1}(\Phi_{1,0}^i)^2+q\chi^2_{\delta_j} U_{1,0}^{q-1}(\Psi_{1,0}^i)^2
  \big)dx+O(\varepsilon),
  \end{align}
as $\varepsilon\rightarrow0$,   uniformly with respect to $\bar{\xi}$ in $\mathcal{M}^k$ and to $\bar{t}$ in compact subsets of $(\mathbb{R}^+)^k$, where  $\chi_{\delta_j}(x)=\chi({\delta_j|x|})$. On the other hand,  it follows from $(\Psi_{\varepsilon,\bar{t},\bar{\xi}},\Phi_{\varepsilon,\bar{t},\bar{\xi}})\in
  \mathcal{Z}_{\bar{\delta},\bar{\xi}}$ that
  \begin{align}\label{take2}
    &\mathcal{J}_\varepsilon'(\mathcal{W}_{\bar{\delta},\bar{\xi}}+\Psi_{\varepsilon,\bar{t},\bar{\xi}},\mathcal{H}_{\delta,\xi}+\Phi_{\varepsilon,\bar{t},\bar{\xi}})(\Psi_{\delta_j,\xi_j}^i,\Phi^i_{\delta_j,\xi_j})
   \nonumber \\=&\int\limits_{\mathcal{M}}\big(-\Delta_g \mathcal{W}_{\bar{\delta},\bar{\xi}}+h\mathcal{W}_{\bar{\delta},\bar{\xi}}-f_\varepsilon(\mathcal{H}_{\bar{\delta},\bar{\xi}})\big)\Phi^i_{\delta_j,\xi_j}d v_g+\int\limits_{\mathcal{M}}\big(-\Delta_g \mathcal{H}_{\bar{\delta},\bar{\xi}}+h\mathcal{H}_{\bar{\delta},\bar{\xi}}-g_\varepsilon(\mathcal{W}_{\bar{\delta},\bar{\xi}})\big)\Psi^i_{\delta_j,\xi_j}d v_g\nonumber\\
    &-\int\limits_{\mathcal{M}}\big(f_\varepsilon(\mathcal{H}_{\bar{\delta},\bar{\xi}}+\Phi_{\varepsilon,\bar{t},\bar{\xi}})-f_\varepsilon(\mathcal{H}_{\bar{\delta},\bar{\xi}})\big)\Phi^i_{\delta_j,\xi_j}dv_g
    -\int\limits_{\mathcal{M}}\big(g_\varepsilon(\mathcal{W}_{\bar{\delta},\bar{\xi}}+\Psi_{\varepsilon,\bar{t},\bar{\xi}})-g_\varepsilon(\mathcal{W}_{\bar{\delta},\bar{\xi}})\big)\Psi^i_{\delta_j,\xi_j}dv_g
 \nonumber \\ \leq &\big\|-\Delta_g \mathcal{W}_{\bar{\delta},\bar{\xi}}+h\mathcal{W}_{\bar{\delta},\bar{\xi}}-f_\varepsilon(\mathcal{H}_{\bar{\delta},\bar{\xi}})\big\|_{\frac{p+1}{p}}\|\Phi^i_{\delta_j,\xi_j}\|_{p+1}
 +\big\|-\Delta_g \mathcal{H}_{\bar{\delta},\bar{\xi}}+h\mathcal{H}_{\bar{\delta},\bar{\xi}}-g_\varepsilon(\mathcal{W}_{\bar{\delta},\bar{\xi}})\big\|_{\frac{q+1}{q}}\|\Psi^i_{\delta_j,\xi_j}\|_{q+1}
 \nonumber \\&+C\|\Phi^i_{\delta_j,\xi_j}\|_{p+1}\|\Phi_{\varepsilon,\bar{t},\bar{\xi}}\|_{\frac{p+1}{1+\alpha\varepsilon}}\sum\limits_{j=1}^k\|H_{\delta_j,\xi_j}\|_{p+1}^{p-1-\alpha\varepsilon}+
  C\|\Phi^i_{\delta_j,\xi_j}\|_{p+1}\|\Phi_{\varepsilon,\bar{t},\bar{\xi}}\|_{\frac{(p-\alpha\varepsilon)(p+1)}{p}}^{p-\alpha\varepsilon}\nonumber\\
  &+C\|\Psi^i_{\delta_j,\xi_j}\|_{q+1}\|\Psi_{\varepsilon,\bar{t},\bar{\xi}}\|_{\frac{q+1}{1+\beta\varepsilon}}\sum\limits_{j=1}^k\|W_{\delta_j,\xi_j}\|_{q+1}^{q-1-\beta\varepsilon}+
  C\|\Psi^i_{\delta_j,\xi_j}\|_{q+1}\|\Psi_{\varepsilon,\bar{t},\bar{\xi}}\|_{\frac{(q-\beta\varepsilon)(q+1)}{q}}^{q-\beta\varepsilon}
  =o(\varepsilon^{\vartheta}),
  \end{align}
as $\varepsilon\rightarrow0$,   uniformly with respect to $\bar{\xi}$ in $\mathcal{M}^k$ and to $\bar{t}$ in compact subsets of $(\mathbb{R}^+)^k$,
  where we have used the fact that $\|\Psi^i_{\delta_j,\xi_j}\|_{q+1}<+\infty$ and $\|\Phi^i_{\delta_j,\xi_j}\|_{p+1}<+\infty$ for any $1<p\leq \frac{N+2}{N-2}\leq q$, $i=0,1,\cdots,N$ and $j=1,2,\cdots,k$. From \eqref{take1} and \eqref{take2}, we prove the claim.

 By \eqref{ch1} and \eqref{ch2}, we can compute
  \begin{align}\label{C11}
  &\partial_{t_m}\widetilde{\mathcal{J}}_{\varepsilon}(\bar{t},\bar{\xi})-\partial_{t_m} \Psi_k(\bar{t},\bar{\xi})\nonumber\\
  =&-\frac{1}{2t_m}\Big(\int\limits_{\mathcal{M}}\big(-\Delta_g \Psi^0_{\delta_m,\xi_m}+h\Psi^0_{\delta_m,\xi_m}-f'_\varepsilon(\mathcal{H}_{\bar{\delta},\bar{\xi}})\Phi^0_{\delta_m,\xi_m}\big)\Phi_{\varepsilon,\bar{t},\bar{\xi}}d v_g\nonumber\\
  &+\int\limits_{\mathcal{M}}\big(-\Delta_g \Phi^0_{\delta_m,\xi_m}+h\Phi^0_{\delta_m,\xi}-g'_\varepsilon(\mathcal{W}_{\bar{\delta},\bar{\xi}})\Psi^0_{\delta_m,\xi_m}\big)\Psi_{\varepsilon,\bar{t},\bar{\xi}}d v_g\nonumber\\
  &-\int\limits_{\mathcal{M}}\big(f_\varepsilon(\mathcal{H}_{\bar{\delta},\bar{\xi}}+\Phi_{\varepsilon,\bar{t},\bar{\xi}})-
  f_\varepsilon(\mathcal{H}_{\bar{\delta},\bar{\xi}})-f'_\varepsilon(\mathcal{H}_{\bar{\delta},\bar{\xi}}) \Phi_{\varepsilon,\bar{t},\bar{\xi}}\big) \Phi^0_{\delta_m,\xi_m}dv_g\nonumber\\
  &-\int\limits_{\mathcal{M}}\big(g_\varepsilon(\mathcal{W}_{\bar{\delta},\bar{\xi}}+\Psi_{\varepsilon,\bar{t},\bar{\xi}})-
  g_\varepsilon(\mathcal{W}_{\bar{\delta},\bar{\xi}})-g'_\varepsilon(\mathcal{W}_{\bar{\delta},\bar{\xi}}) \Psi_{\varepsilon,\bar{t},\bar{\xi}}\big) \Psi^0_{\delta_m,\xi_m}dv_g\Big)\nonumber\\
  &+ \mathcal{J}_\varepsilon'(\mathcal{W}_{\bar{\delta},\bar{\xi}}+\Psi_{\varepsilon,\bar{t},\bar{\xi}},\mathcal{H}_{\bar{\delta},\bar{\xi}}
   +\Phi_{\varepsilon,\bar{t},\bar{\xi}})\big(\partial _{t_m}\Psi_{\varepsilon,\bar{t},\bar{\xi}},\partial _{t_m}\Phi_{\varepsilon,\bar{t},\bar{\xi}}\big),
\end{align}
and
\begin{align}\label{C12}
  &\partial_{y^m_l}\widetilde{\mathcal{J}}_{\varepsilon}(\bar{t},\bar{\xi}(y))\big|_{y=0}-\partial_{y^m_l}\Psi_k(\bar{t},\bar{\xi}(y))\big|_{y=0}\nonumber\\
  =&\frac{1}{\delta_m}\Big(\int\limits_{\mathcal{M}}\big(-\Delta_g \Psi^l_{\delta_m,\xi_m}+h\Psi^l_{\delta_m,\xi_m}-f'_\varepsilon(\mathcal{H}_{\bar{\delta},\bar{\xi}})\Phi^l_{\delta_m,\xi_m}\big)\Phi_{\varepsilon,\bar{t},\bar{\xi}}d v_g\nonumber\\
  &+\int\limits_{\mathcal{M}}\big(-\Delta_g \Phi^l_{\delta_m,\xi_m}+h\Phi^l_{\delta_m,\xi_m}-g'_\varepsilon(\mathcal{W}_{\bar{\delta},\bar{\xi}})\Psi^l_{\delta_m,\xi_m}\big)\Psi_{\varepsilon,\bar{t},\bar{\xi}}d v_g\nonumber\\
  &-\int\limits_{\mathcal{M}}\big(f_\varepsilon(\mathcal{H}_{\bar{\delta},\bar{\xi}}+\Phi_{\varepsilon,\bar{t},\bar{\xi}})-
  f_\varepsilon(\mathcal{H}_{\bar{\delta},\bar{\xi}})-f'_\varepsilon(\mathcal{H}_{\bar{\delta},\bar{\xi}}) \Phi_{\varepsilon,\bar{t},\bar{\xi}}\big) \Phi^l_{\delta_m,\xi_m}dv_g\nonumber\\
  &-\int\limits_{\mathcal{M}}\big(g_\varepsilon(\mathcal{W}_{\bar{\delta},\bar{\xi}}+\Psi_{\varepsilon,\bar{t},\bar{\xi}})-
  g_\varepsilon(\mathcal{W}_{\bar{\delta},\bar{\xi}})-g'_\varepsilon(\mathcal{W}_{\bar{\delta},\bar{\xi}}) \Psi_{\varepsilon,\bar{t},\bar{\xi}}\big) \Psi^l_{\delta_m,\xi_m}dv_g\Big)\nonumber\\
  &+ \mathcal{J}_\varepsilon'(\mathcal{W}_{\bar{\delta},\bar{\xi}}+\Psi_{\varepsilon,\bar{t},\bar{\xi}},\mathcal{H}_{\bar{\delta},\bar{\xi}}
   +\Phi_{\varepsilon,\bar{t},\bar{\xi}})
   \big(\partial _{y^m_l}\Psi_{\varepsilon,\bar{t},\bar{\xi}(y)}\big|_{y=0},\partial _{y^m_l}\Phi_{\varepsilon,\bar{t},\bar{\xi}(y)}\big|_{y=0}\big)+o(\varepsilon^{\frac{3\vartheta}{2}}),
 \end{align}
 as $\varepsilon\rightarrow0$.
 Next, we estimate \eqref{C11} and \eqref{C12}. By the H\"{o}lder inequality, Proposition \ref{propo1},
   and the Sobolev embedding theorem, arguing as Lemma \ref{error}, for any $l=0,1,\cdots,N$,  we have
\begin{align*}
  &\int\limits_{\mathcal{M}}\big(-\Delta_g \Psi^l_{\delta_m,\xi_m}+h\Psi^l_{\delta_m,\xi_m}-f'_\varepsilon(\mathcal{H}_{\bar{\delta},\bar{\xi}})\Phi^l_{\delta_m,\xi_m}\big)\Phi_{\varepsilon,\bar{t},\bar{\xi}}d v_g
  \\ \leq&
  \big\|-\Delta_g \Psi^l_{\delta_m,\xi_m}+h\Psi^l_{\delta_m,\xi_m}-f'_\varepsilon(\mathcal{H}_{\bar{\delta},\bar{\xi}})\Phi^l_{\delta_m,\xi_m}\big\|_{\frac{p+1}{p}}\|\Phi_{\varepsilon,\bar{t},\bar{\xi}}\|_{p+1}=o(\varepsilon^{2\vartheta}),
\end{align*}
and
\begin{align*}
  &\int\limits_{\mathcal{M}}\big(-\Delta_g \Phi^l_{\delta_m,\xi_m}+h\Phi^l_{\delta_m,\xi_m}-g'_\varepsilon(\mathcal{W}_{\bar{\delta},\bar{\xi}})\Psi^l_{\delta_m,\xi_m}\big)\Psi_{\varepsilon,\bar{t},\bar{\xi}}d v_g
  \\ \leq&
  \big\|-\Delta_g \Phi^l_{\delta_m,\xi_m}+h\Phi^l_{\delta_m,\xi_m}-g'_\varepsilon(\mathcal{W}_{\bar{\delta},\bar{\xi}})\Psi^l_{\delta_m,\xi_m}\big\|_{\frac{q+1}{q}}\|\Psi_{\varepsilon,\bar{t},\bar{\xi}}\|_{q+1}=o(\varepsilon^{2\vartheta}),
\end{align*}
as $\varepsilon\rightarrow0$, uniformly with respect to $\bar{\xi}$ in $\mathcal{M}^k$ and to $\bar{t}$ in compact subsets of $(\mathbb{R}^+)^k$.
Moreover, by the mean value formula, Lemma \ref{gs}, \eqref{ener5}, \eqref{ener7} and the Sobolev embedding theorem, for any $l=0,1,\cdots,N$,  we obtain
\begin{align*}
  &\int\limits_{\mathcal{M}}\big(f_\varepsilon(\mathcal{H}_{\bar{\delta},\bar{\xi}}+\Phi_{\varepsilon,\bar{t},\bar{\xi}})-
  f_\varepsilon(\mathcal{H}_{\bar{\delta},\bar{\xi}})-f'_\varepsilon(\mathcal{H}_{\bar{\delta},\bar{\xi}}) \Phi_{\varepsilon,\bar{t},\bar{\xi}}\big) \Phi^l_{\delta_m,\xi_m}dv_g\\
  \leq&
  C\int\limits_{\mathcal{M}}\mathcal{H}_{\bar{\delta},\bar{\xi}}^{p-2-\alpha\varepsilon}\Phi_{\varepsilon,\bar{t},\bar{\xi}}^{2}\Phi^l_{\delta_m,\xi_m}dv_g \leq C\|\Phi^l_{\delta_m,\xi_m}\|_{p+1}\|\Phi_{\varepsilon,\bar{t},\bar{\xi}}\|_{p+1}^{2}\sum\limits_{j=1}^k\|H_{\delta_j,\xi_j}\|_{\frac{(p-2-\alpha\varepsilon)(p+1)}{p-2}}^{p-2-\alpha\varepsilon}=o(\varepsilon^{2\vartheta}),
\end{align*}
and
\begin{align*}
  &\int\limits_{\mathcal{M}}\big(g_\varepsilon(\mathcal{W}_{\bar{\delta},\bar{\xi}}+\Psi_{\varepsilon,\bar{t},\bar{\xi}})-
  g_\varepsilon(\mathcal{W}_{\bar{\delta},\bar{\xi}})-g'_\varepsilon(\mathcal{W}_{\bar{\delta},\bar{\xi}}) \Psi_{\varepsilon,\bar{t},\bar{\xi}}\big)\Psi^l_{\delta_m,\xi_m}dv_g
  \\ \leq& \left\{
    \begin{array}{ll}
   \displaystyle C\int\limits_{\mathcal{M}}\mathcal{W}_{\bar{\delta},\bar{\xi}}^{q-2-\beta\varepsilon}\Psi^2_{\varepsilon,\bar{t},\bar{\xi}}\Psi^l_{\delta_m,\xi_m}dv_g+
  C\int\limits_{\mathcal{M}}\Psi_{\varepsilon,\bar{t},\bar{\xi}}^{q-\beta\varepsilon}\Psi^l_{\delta_m,\xi_m}dv_g,\quad &\text{if $q>2$},\\
  \displaystyle C\int\limits_{\mathcal{M}}\mathcal{W}_{\bar{\delta},\bar{\xi}}^{q-2-\beta\varepsilon}\Psi^2_{\varepsilon,\bar{t},\bar{\xi}}\Psi^l_{\delta_m,\xi_m}dv_g,\quad &\text{if $q\leq 2$},
  \end{array}
  \right.
  \\ \leq& \left\{
    \begin{array}{ll}
   \displaystyle C\|\Psi^l_{\delta_m,\xi_m}\|_{q+1}\|\Psi_{\varepsilon,\bar{t},\bar{\xi}}\|_{q+1}^{2}\sum\limits_{j=1}^k\|W_{\delta_j,\xi_j}\|_{\frac{(q-2-\beta\varepsilon)(q+1)}{q-2}}^{q-2-\beta\varepsilon}+
   \|\Psi^l_{\delta_m,\xi_m}\|_{q+1}\|\Psi_{\varepsilon,\bar{t},\bar{\xi}}\|_{\frac{(q-\beta\varepsilon)(q+1)}{q}}^{q-\beta\varepsilon},\quad &\text{if $q>2$},\\
  C\|\Psi^l_{\delta_m,\xi_m}\|_{q+1}\|\Psi_{\varepsilon,\bar{t},\bar{\xi}}\|_{q+1}^{2}\sum\limits_{j=1}^k\|W_{\delta_j,\xi_j}\|_{\frac{(q-2-\beta\varepsilon)(q+1)}{q-2}}^{q-2-\beta\varepsilon},\quad&
  \text{if $q\leq 2$},
  \end{array}
  \right.
  \\
  =&o(\varepsilon^{2\vartheta}),
\end{align*}
as $\varepsilon\rightarrow0$, uniformly with respect to $\bar{\xi}$ in $\mathcal{M}^k$ and to $\bar{t}$ in compact subsets of $(\mathbb{R}^+)^k$.
Using \eqref{embedd}, \eqref{e54}, \eqref{e55}, \eqref{real} and \eqref{cla}, for any $l=1,2,\cdots,N$,  we get
\begin{align*}
  &\mathcal{J}_\varepsilon'(\mathcal{W}_{\bar{\delta},\bar{\xi}}+\Psi_{\varepsilon,\bar{t},\bar{\xi}},\mathcal{H}_{\bar{\delta},\bar{\xi}}
   +\Phi_{\varepsilon,\bar{t},\bar{\xi}})\big(\partial _{t_m}\Psi_{\varepsilon,\bar{t},\bar{\xi}},\partial _{t_m}\Phi_{\varepsilon,\bar{t},\bar{\xi}}\big)
 \nonumber\\  =&\sum\limits_{i=0}^N\sum\limits_{j=1}^kc_{ij}\big\langle\big(\Psi_{\delta_j,\xi_j}^i,\Phi^i_{\delta_j,\xi_j}\big),\big(\partial _{t_m}\Psi_{\varepsilon,\bar{t},\bar{\xi}},\partial _{t_m}\Phi_{\varepsilon,\bar{t},\bar{\xi}}\big)\big\rangle_h\nonumber\\
 =&-\sum\limits_{i=0}^N\sum\limits_{j=1}^kc_{ij}\big\langle\big(\partial _{t_m}\Psi_{\delta_j,\xi_j}^i,\partial _{t_m}\Phi^i_{\delta_j,\xi_j}\big),\big(\Psi_{\varepsilon,\bar{t},\bar{\xi}},\Phi_{\varepsilon,\bar{t},\bar{\xi}}\big)\big\rangle_h
 \\=&o\Big(\varepsilon^\vartheta\sum\limits_{i=0}^N\sum\limits_{i=1}^k|c_{ij}|\Big)=o(\varepsilon^{2\vartheta}),
\end{align*}
and
\begin{align*}
  &\mathcal{J}_\varepsilon'(\mathcal{W}_{\bar{\delta},\bar{\xi}}+\Psi_{\varepsilon,\bar{t},\bar{\xi}},\mathcal{H}_{\bar{\delta},\bar{\xi}}
   +\Phi_{\varepsilon,\bar{t},\bar{\xi}})
   \big(\partial _{y^m_l}\Psi_{\varepsilon,\bar{t},\bar{\xi}(y)}\big|_{y=0},\partial _{y^m_l}\Phi_{\varepsilon,\bar{t},\bar{\xi}(y)}\big|_{y=0}\big)\nonumber\\
   =&
   \sum\limits_{i=0}^N\sum\limits_{j=1}^kc_{ij}\big\langle\big(\Psi_{\delta_j,\xi_j}^i,\Phi^i_{\delta_j,\xi_j}\big),\big(\partial _{y^m_l}\Psi_{\varepsilon,\bar{t},\bar{\xi}(y)}\big|_{y=0},\partial _{y^m_l}\Phi_{\varepsilon,\bar{t},\bar{\xi}(y)}\big|_{y=0}\big)\big\rangle_h\nonumber\\
   =&-\sum\limits_{i=0}^N\sum\limits_{j=1}^kc_{ij}\big\langle\big(\partial _{y^m_l}\Psi_{\delta_j,\xi_j}^i\big|_{y=0},\partial _{y^m_l}\Phi^i_{\delta_j,\xi_j}\big|_{y=0}\big),
   \big(\Psi_{\varepsilon,\bar{t},\bar{\xi}(y)},\Phi_{\varepsilon,\bar{t},\bar{\xi}(y)}\big)\big\rangle_h\\
   =&o\Big(\varepsilon^{\frac{2\vartheta-1}{2}}\sum\limits_{i=0}^N\sum\limits_{j=1}^k|c_{ij}|\Big)=o\big(\varepsilon^{\frac{4\vartheta-1}{2}}\big),
\end{align*}
as $\varepsilon\rightarrow0$, uniformly with respect to $\bar{\xi}$ in $\mathcal{M}^k$ and to $\bar{t}$ in compact subsets of $(\mathbb{R}^+)^k$.
Taking $\frac{3}{4}<\vartheta<1$, we complete the proof.
\end{proof}

\end{document}